\documentclass[12pt, reqno]{amsart}
\usepackage{amssymb}
\usepackage{eucal}
\usepackage{amsmath}
\usepackage{amscd}
\usepackage[dvips]{color}
\usepackage{mathtools}
\usepackage{multicol}
\usepackage[all]{xy}           %xypic macro for latex2.09
\usepackage{graphicx}
\usepackage{color}
\usepackage{colordvi}
\usepackage{xspace}
\usepackage{ulem}
\usepackage{cancel}
\usepackage{tikz}
\usepackage{longtable}
\usepackage{multirow}
\usepackage{ifpdf}
\ifpdf
 \usepackage[colorlinks,final,backref=page,hyperindex]{hyperref}
\else
 \usepackage[colorlinks,final,backref=page,hyperindex,hypertex]{hyperref}
\fi

\begin{document}
\newcommand {\emptycomment}[1]{} %to remove paragraphs

\newcommand{\tabincell}[2]{\begin{tabular}{@{}#1@{}}#2\end{tabular}}

\newcommand{\nc}{\newcommand}
\newcommand{\delete}[1]{}

%%%%%%Use the next black to suppress labels
%\delete{
\nc{\mlabel}[1]{\label{#1}}  % Use this to suppress names
\nc{\mcite}[1]{\cite{#1}}  % Use this to suppress names
\nc{\mref}[1]{\ref{#1}}  % Use this to suppress names
\nc{\meqref}[1]{Eq.~\eqref{#1}} % Use this to suppress names
\nc{\mbibitem}[1]{\bibitem{#1}} % Use this to show number
%}

%%%%%%%%%%%Use the next block to show labels
\delete{
\nc{\mlabel}[1]{\label{#1}  % Use the next two lines to show names
{\hfill \hspace{1cm}{\bf{{\ }\hfill(#1)}}}}
\nc{\mcite}[1]{\cite{#1}{{\bf{{\ }(#1)}}}}  % Use this lines to show names
\nc{\mref}[1]{\ref{#1}{{\bf{{\ }(#1)}}}}  % Use this lines to show names
\nc{\meqref}[1]{Eq.~\eqref{#1}{{\bf{{\ }(#1)}}}} % Use this lines to show names
\nc{\mbibitem}[1]{\bibitem[\bf #1]{#1}} % Use this to show name
}

%%%%%%%%%%%%%%%%%%%%%%%% Statements
\newtheorem{thm}{Theorem}[section]
\newtheorem{lem}[thm]{Lemma}
\newtheorem{cor}[thm]{Corollary}
\newtheorem{pro}[thm]{Proposition}
\newtheorem{conj}[thm]{Conjecture}
\theoremstyle{definition}
\newtheorem{defi}[thm]{Definition}
\newtheorem{ex}[thm]{Example}
\newtheorem{rmk}[thm]{Remark}
\newtheorem{pdef}[thm]{Proposition-Definition}
\newtheorem{condition}[thm]{Condition}

\renewcommand{\labelenumi}{{\rm(\alph{enumi})}}
\renewcommand{\theenumi}{\alph{enumi}}
\renewcommand{\labelenumii}{{\rm(\roman{enumii})}}
\renewcommand{\theenumii}{\roman{enumii}}

\nc{\tred}[1]{\textcolor{red}{#1}}
\nc{\tblue}[1]{\textcolor{blue}{#1}}
\nc{\tgreen}[1]{\textcolor{green}{#1}}
\nc{\tpurple}[1]{\textcolor{purple}{#1}}
\nc{\btred}[1]{\textcolor{red}{\bf #1}}
\nc{\btblue}[1]{\textcolor{blue}{\bf #1}}
\nc{\btgreen}[1]{\textcolor{green}{\bf #1}}
\nc{\btpurple}[1]{\textcolor{purple}{\bf #1}}

%new commands

\newcommand{\End}{\text{End}}

\nc{\calb}{\mathcal{B}}
\nc{\call}{\mathcal{L}}
\nc{\calo}{\mathcal{O}}
\nc{\frakg}{\mathfrak{g}}
\nc{\frakh}{\mathfrak{h}}
\nc{\ad}{\mathrm{ad}}

\nc{\ccred}[1]{\tred{\textcircled{#1}}}

%\nc{\move}[1]{\footnote{#1}}

\newcommand{\cm}[1]{\textcolor{purple}{\underline{CM:}#1 }}
\newcommand{\gl}[1]{\textcolor{blue}{\underline{GL:}#1 }}
\newcommand{\red}[1]{\textcolor{red}{#1}}

%%%%%%%%%%%%%%%%%%%%%%%%%%%%%%%%%%%%%%%%%%%%%%%%%%%%%%%%%%%%%%%%%%
\title[Quasi-triangular Leibniz bialgebras]%yemei
{Quasi-triangular, factorizable Leibniz bialgebras and relative Rota-Baxter operators}%first page

\author{Chengming Bai}
\address{Chern Institute of Mathematics \& LPMC, Nankai University, Tianjin 300071, China}
\email{baicm@nankai.edu.cn}

\author{Guilai Liu}
\address{Chern Institute of Mathematics \& LPMC, Nankai University, Tianjin 300071, China}
\email{1120190007@mail.nankai.edu.cn}

\author{Yunhe Sheng}
\address{Department of Mathematics, Jilin University, Changchun 130012, Jilin, China}
\email{shengyh@jlu.edu.cn}

\author{Rong Tang}
\address{Department of Mathematics, Jilin University, Changchun 130012, Jilin, China}
\email{tangrong@jlu.edu.cn}

%\date{\today}%

\begin{abstract}
We introduce the notion of quasi-triangular Leibniz bialgebras,
which {{can be}}  constructed from solutions of the classical
Leibniz Yang-Baxter equation (CLYBE) whose skew-symmetric parts
are invariant.
{{In addition to}} triangular Leibniz bialgebras, %studied in  \cite{ST},
quasi-triangular Leibniz bialgebras contain factorizable Leibniz
bialgebras as another  subclass, which lead to a factorization of
the underlying Leibniz algebras. Relative Rota-Baxter operators
with weights on Leibniz algebras are used to  characterize
solutions of the CLYBE whose skew-symmetric parts are invariant.
{{On skew-symmetric quadratic Leibniz algebras, such operators
correspond to Rota-Baxter type operators. Consequently, we
introduce the notion of skew-symmetric quadratic Rota-Baxter
Leibniz algebras, such that they give rise to triangular Leibniz
bialgebras in the case of weight $0$, while they are in one-to-one
correspondence with factorizable Leibniz bialgebras in the case of
nonzero weights. }}
% and in particular, lead to the notion of skew-symmetric quadratic Rota-Baxter Leibniz algebras.
%Skew-symmetric quadratic Rota-Baxter Leibniz algebras of weight $0$ give rise to  triangular Leibniz bialgebras,
%{{while}} skew-symmetric quadratic Rota-Baxter Leibniz algebras of nonzero weights are in one-to-one correspondence with factorizable Leibniz bialgebras.
\end{abstract}

\subjclass[2020]{
    17A32, %Leibniz algebras
    17A36,  %Automorphisms, derivations, other operators (nonassociative rings and algebras)
    %17A40,  %Ternary compositions
    %18M70,  %Algebraic operads, cooperads, and Koszul duality
    %17B:Lie algebras and Lie superalgebras
    17B10, %Representations, algebraic theory
    17B38, %Yang-Baxter equations and Rota-Baxter operators
    17B40, %Automorphisms,derivations,other operators
    17B60, %Lie (super)algebras associated with other structures (associative, Jordan, ect.)
    17B62. % Lie bialgebras; Lie coalgebras
    %17B63,  %Poisson algebras
    %17D25.  %Lie-admissible algebras
    %37J39,   %Relations of finite-dimensional Hamiltonian and Lagrangian systems with topology, geometry and differential geometry (symplectic geometry, Poisson geometry, etc.
    %53D17  %Poisson manifolds; Poisson groupoids and algebroids
}

\keywords{Quasi-triangular Leibniz bialgebras, classical Leibniz Yang-Baxter equation, factorizable Leibniz bialgebras, (relative) Rota-Baxter operators}

\maketitle

%\vspace{-1.5cm}

\tableofcontents

\allowdisplaybreaks

\section{Introduction}

Leibniz algebras, first discovered by Bloh under the name of $D$-algebras \cite{Blo}, were rediscovered and studied in \cite{Lod2} as noncommutative analogues of Lie algebras.
\begin{defi}
    A \textbf{(left) Leibniz algebra} is a pair $(A,[-,-]_{A})$, where $A$ is a vector space and $[-,-]_{A}:A\otimes A\rightarrow A$ is a multiplication such that
    \begin{equation}\label{eq:defi:Leibniz}
        [x,[y,z]_{A}]_{A}=[[x,y]_{A},z]_{A}+[y,[x,z]_{A}]_{A},\;\;\forall x,y,z\in A.
    \end{equation}
\end{defi}
 As a generalization of Lie algebras, Leibniz algebras form an important class of non-associative algebras \cite{AKOZ,DMS,Gne}. The (co)homology and homotopy theories of Leibniz algebras were established in \cite{ammardefiLeibnizalgebra,Gao,Lodder,Lod2,Pir}. Recently,
Leibniz algebras were studied in different aspects due to their various applications in mathematics and physics, such as integration \cite{Bor,Cov,JP}, deformation quantization \cite{Dhe}, rational homotopy theory \cite{Livernet},  higher-order differential operators \cite{Akman} and higher gauge theories \cite{Kot, Str}. From the operadic viewpoint, the operad of Leibniz algebras is the duplicator of the operad of  Lie algebras \cite{Pei}.

%\subsection{Quasi-triangular Leibniz bialgebras and generalized $\mathcal{O}$-operators}\
\subsection{The previous study on Leibniz bialgebras}\

A bialgebra structure is a vector space equipped with an algebra structure and a coalgebra structure satisfying some compatible conditions.
Some well-known bialgebra structures are Lie bialgebras \cite{Cha,Dri}  and  antisymmetric  infinitesimal bialgebras
\cite{Agu2000, Agu2001, Agu2004, Bai2010} as a subclass of infinitesimal bialgebras. These bialgebras
have a common property, that is, they have equivalent characterizations in terms of Manin triples which correspond to  nondegenerate  bilinear forms on the algebra structures satisfying certain conditions. In addition, it is noteworthy that one can  study the bialgebra theories in the sense that they give a distributive law
 in \cite{LMW,Lod3}.

\delete{
Due to the close relationship between Leibniz algebras and Lie algebras,
the bialgebra theory for Leibniz algebras was considered respectively in the assumptions that the co-multiplication is a 1-cocycle in \cite{Rez} and that the invariant condition of  Manin triples of Leibniz algebras takes the same form as that of Lie algebras
in \cite{Bar}. For the latter case, the Leibniz algebras with such nondegenerate bilinear forms are restricted to symmetric Leibniz algebras which are both left and right Leibniz algebras, and consequently the study in \cite{Bar} mainly concentrated on symmetric Leibniz bialgebras.}

In
\cite{Chap}, Chapton used the operad theory to determine the proper invariant condition of bilinear forms on Leibniz algebras.
He revealed an impressive difference between Lie algebras and Leibniz algebras: the operad of Lie algebras is cyclic, whereas the operad of Leibniz algebras is anticyclic. Consequently, the invariant bilinear forms on Leibniz algebras should be skew-symmetric.

%\cm{I think that we should unify ``skew-symmetric" and ``antisymmetric". Maybe we use antisymmetric to replace skew-symmetric since anticyclic is used.}

\begin{defi}\cite{Chap}
A \textbf{skew-symmetric quadratic Leibniz algebra} is a triple
$(A$,
$[-,-]_{A}$,
$\omega)$,
where $(A,[-,-]_{A})$ is a Leibniz algebra and $\omega$ is a
nondegenerate skew-symmetric bilinear form on $A$ which is invariant in the sense that
\begin{equation}\label{eq:inv bilinear form}
    \omega(x,[y,z]_{A})=\omega([x,z]_{A}+[z,x]_{A},y),\;\;\forall x,y,z\in A.
\end{equation}
\end{defi}
With such  bilinear forms, a notion of Leibniz bialgebras was introduced in \cite{ST}, and it was shown that Leibniz bialgebras and  Manin triples of Leibniz algebras  are equivalent structures. See \cite{ABBBS,Barreiro-Benayadi,RSH} for other approaches to the study of Leibniz bialgebras and applications. In particular, \cite{ST} used the  twisting theories and relative Rota-Baxter operators to  study triangular Leibniz bialgebras. However, quasi-triangular Leibniz bialgebras are yet to be studied in \cite{ST}.

\subsection{Our {{further}} study on Leibniz bialgebras}\

In this paper, we reconsider the theory of Leibniz bialgebras, especially quasi-triangular Leibniz bialgebras and the
classical Leibniz Yang-Baxter equation (CLYBE).
%Given a Leibniz algebra $(A,[-,-]_{A})$, we introduce the  invariant condition for a 2-tensor $r\in A\otimes A$.
We introduce an invariant condition for a 2-tensor in a Leibniz
algebra. By introducing a {{suitable}} co-multiplication, we
directly show that a solution of the CLYBE whose skew-symmetric
part is invariant gives rise to a Leibniz bialgebra that we call
quasi-triangular. {{Comparing to \cite{ST} which focuses on
triangular Leibniz bialgebras, our new approach has the advantage
that the solutions of the CLYBE giving rise to Leibniz bialgebras
are no longer limited to be symmetric.} \delete{ Such a method is
different from \cite{ST} and indicates a priority since it
generalizes the result obtained therein that a symmetric solution
of the CLYBE  gives  a triangular Leibniz bialgebra.}

{{Other than}} triangular Leibniz bialgebras, there is another
subclass of quasi-triangular Leibniz bialgebras, namely
factorizable Leibniz bialgebras.
%The double Leibniz algebra of a factorizable Leibniz bialgebra is isomorphic to a direct sum of the underlying Leibniz algebras. In particular,
Factorizable Leibniz bialgebras give rise to a natural factorization of the underlying Leibniz algebras. %, which justifies their name.
The importance of factorizable Leibniz bialgebras in the study of Leibniz bialgebras can also be observed from the fact that the double space of an arbitrary Leibniz bialgebra admits a factorizable Leibniz bialgebra structure.

%As the name suggests, factorizable Leibniz bialgebras give rise to a natural factorization of the underlying Leibniz algebras. Factorizable Leibniz bialgebras serve as an important part in the study of Leibniz bialgebras,
%: On the one hand, the Leibniz algebra structure on the double space $A\oplus A^{*}$ of a factorizable Leibniz bialgebra $(A,[-,-]_{A},\Delta_{r})$ is isomorphic to the direct sum $A\oplus A$ of Leibniz algebras. On the other hand,
%since the  double space of an arbitrary Leibniz bialgebra admits a factorizable Leibniz bialgebra structure.

Furthermore, we study solutions of the CLYBE in Leibniz algebras
in terms of operator forms. Solutions of the CLYBE whose
skew-symmetric parts are invariant can be characterized as
relative Rota-Baxter operators with weights on Leibniz algebras
\cite{Das}. {{  Then we introduce the notion of skew-symmetric
quadratic Rota-Baxter
        Leibniz algebras, which have close relationships with triangular and factorizable Leibniz bialgebras. Refer to Corollary \ref{pro:triangular Leib}, Theorem \ref{thm:quadratic to fact} and Theorem \ref{thm:fact to quadratic} for more detail.
}}
\delete{
In particular, a skew-symmetric quadratic Rota-Baxter
Leibniz algebra of weight 0 gives rise to a triangular Leibniz
bialgebra, and there is a one-to-one correspondence between skew-symmetric
quadratic Rota-Baxter Leibniz algebras of nonzero weights and
factorizable Leibniz bialgebras.
}

This paper is organized as follows.
In Section \ref{sec2}, we show that a solution of the CLYBE in a Leibniz algebra whose skew-symmetric part is invariant  gives rise to a quasi-triangular Leibniz bialgebra.
%In Section \ref{sec2}, we introduce the notion of quasi-triangular Leibniz bialgebras, and show that a solution of the CLYBE in a Leibniz algebra whose skew-symmetric part is invariant  gives rise to a quasi-triangular Leibniz bialgebra.
In Section \ref{sec3},
we study factorizable Leibniz bialgebras which compose a typical subclass of quasi-triangular Leibniz bialgebras.
In Section \ref{sec4}, we study relative Rota-Baxter operators with weights on Leibniz algebras, which serve as operator forms of solutions of the CLYBE whose skew-symmetric parts are invariant.

Throughout this paper,  unless otherwise specified, all the vector
spaces and algebras are finite-dimensional over an algebraically
closed field $\mathbb {K}$ of characteristic zero, although many
results and notions remain valid in the infinite-dimensional case.

\delete{
\subsection{Outline of the paper}\

This paper is organized as follows.

In Section \ref{sec2}, we introduce the  invariant condition for 2-tensors in Leibniz algebras, and show that a solution of the CLYBE in a Leibniz algebra whose skew-symmetric part is invariant  gives rise to a quasi-triangular Leibniz bialgebra.

In Section \ref{sec3},
we introduce the notion of factorizable Leibniz bialgebras, which serve as another typical subclass of quasi-triangular Leibniz bialgebras besides \cm{the} triangular ones. Factorizable Leibniz bialgebras lead to a factorization of the underlying Leibniz algebras, and exist naturally on the double spaces of Leibniz bialgebras.

In Section \ref{sec4}, we study relative Rota-Baxter operators of Leibniz algebras with weights, which serve as operator forms of solutions of the CLYBE whose skew-symmetric parts are invariant.
Such operators on skew-symmetric quadratic Leibniz algebras translate into Rota-Baxter type operators, and they coincide with Rota-Baxter operators on Leibniz algebras in particular cases, which lead to the notion of skew-symmetric quadratic Rota-Baxter Leibniz algebras.
We show that a skew-symmetric quadratic Rota-Baxter Leibniz algebra of weight 0 gives rise to a triangular Leibniz bialgebra, and skew-symmetric quadratic Rota-Baxter Leibniz algebras of nonzero weights are in one-to-one correspondence with factorizable Leibniz bialgebras.

Throughout this paper,  unless otherwise specified, all the vector
spaces and algebras are finite-dimensional over an algebraically
closed field $\mathbb {K}$ of characteristic zero, although many
results and notions remain valid in the infinite-dimensional case.}

\section{Quasi-triangular Leibniz bialgebras and the classical Leibniz Yang-Baxter equation}\label{sec2}

In this section, we introduce the  invariant condition for
2-tensors in Leibniz algebras. Then we introduce a {{suitable}}
co-multiplication on a Leibniz algebra, and show that {{such a
co-multiplication determined by}} a solution of the classical
Leibniz Yang-Baxter equation whose skew-symmetric part is
invariant gives rise to a Leibniz bialgebra  that we call
quasi-triangular.

\begin{defi}
    Let $(A,[-,-]_{A})$ be a Leibniz algebra  and $r=\sum\limits_{i} a_{i}\otimes b_{i}\in A\otimes A$.
    Set
    \begin{equation} \label{eq:admissible}
        [[r,r]]:=[r_{12},r_{13}]-([r_{12},r_{23}]+[r_{23},r_{12}])+[r_{23},r_{13}],
    \end{equation}
where
    \begin{eqnarray*}
        [r_{12},r_{13}]=\sum_{i,j}[a_{i}, a_{j}]_{A}\otimes b_{i}\otimes b_{j},\;
        [r_{12}, r_{23}]=\sum_{i,j}a_{i}\otimes [b_{i},a_{j}]_{A}\otimes b_{j},\;
    \end{eqnarray*}
\begin{eqnarray*}
        [r_{23},r_{12}]=\sum_{i,j}a_{i}\otimes [a_{j},b_{i}]_{A}\otimes b_{j},\;
        [r_{23},r_{13}]=\sum_{i,j}a_{i}\otimes a_{j}\otimes [b_{j},b_{i}]_{A}.
    \end{eqnarray*}
Then $r$ is called a solution of the \textbf{classical Leibniz Yang-Baxter equation} or the \textbf{CLYBE} in short if $[[r,r]]=0$.
\end{defi}

\begin{rmk}
When $r$ is symmetric, we recover the form of the CLYBE which was given in \cite{ST}.
\end{rmk}

For a vector space $V$, let $\sigma\in\mathrm{End}(V\otimes V)$ and $\sigma_{13}\in\mathrm{End}(V\otimes V\otimes V)$ be linear maps defined  as
\begin{equation*}
\sigma(x\otimes y)=y\otimes x,\; \sigma_{13}(x\otimes y\otimes z)=z\otimes y\otimes x,\;\forall x,y,z\in V.
\end{equation*}

\begin{pro}\label{pro:equation equivalence}
    Let $(A,[-,-]_{A})$ be a Leibniz algebra  and $r=\sum\limits_{i}a_{i}\otimes b_{i}\in A\otimes A$ be a $2$-tensor.
    Then we have
    \begin{equation*}
        [[\sigma(r),\sigma(r)]]= \sigma_{13}[[r,r]].
    \end{equation*}
Consequently, $r$ is a solution of the CLYBE in $(A,[-,-]_{A})$ if and only if
    $\sigma(r)$ is a solution of the CLYBE in $(A,[-,-]_{A})$.
    \end{pro}
\begin{proof}
We observe that
\begin{eqnarray*}
[[\sigma(r),\sigma(r)]]&=&\sum_{i,j} [b_{i},b_{j}]_{A}\otimes a_{i}\otimes a_{j}-b_{i}\otimes([b_{j},a_{i}]_{A}+[a_{i},b_{j}]_{A})\otimes a_{j}+b_{i}\otimes b_{j}\otimes[a_{j},a_{i}]_{A}\\
&=&\sigma_{13}[[r,r]].
\end{eqnarray*}
Hence the conclusion follows.
\end{proof}

Let $(A,[-,-]_{A})$ be a Leibniz algebra. Let $\mathcal{L}_{A},\mathcal{R}_{A}:A\rightarrow\mathrm{End}(A)$ be linear maps defined as
$$\mathcal{L}_{A}(x)y=[x,y]_{A}=\mathcal{R}_{A}(y)x,\;\;\forall x,y\in A.$$

We recall the notions of Leibniz coalgebras and Leibniz bialgebras.

\begin{defi}\cite{ST}
A \textbf{Leibniz coalgebra} is a pair $(A,\Delta)$, where $A$ is a vector space and $\Delta:A\rightarrow A\otimes A$ is a co-multiplication such that
\begin{equation}\label{eq:defi:Leibniz co}
    (\Delta\otimes\mathrm{id})\Delta(x)+(\sigma\otimes\mathrm{id})(\mathrm{id}\otimes\Delta)\Delta(x)-(\mathrm{id}\otimes\Delta)\Delta(x)=0,\;\;\forall x\in A.
\end{equation}
A \textbf{Leibniz bialgebra} is a triple $(A,[-,-]_{A},\Delta)$, where $(A,[-,-]_{A})$ is a Leibniz algebra, $(A,\Delta)$ is a Leibniz coalgebra and the following equations hold:
\begin{small}
    \begin{equation}\label{eq:defi:Leibniz bialg1}
        \sigma(\mathcal{R}_{A}(y)\otimes\mathrm{id})\Delta(x)=(\mathcal{R}_{A}(x)\otimes\mathrm{id})\Delta(y),
    \end{equation}
    \begin{equation}\label{eq:defi:Leibniz bialg2}
        \Delta([x,y]_{A})=(\mathrm{id}\otimes\mathcal{R}_{A}(y)-(\mathcal{L}_{A}+\mathcal{R}_{A})(y)\otimes\mathrm{id})(\mathrm{id}+\sigma)\Delta(x)+(\mathrm{id}\otimes\mathcal{L}_{A}(x)+\mathcal{L}_{A}(x)\otimes\mathrm{id})\Delta(y),
    \end{equation}
\end{small}
for all $x,y\in A$.
\end{defi}

\delete{
\begin{defi}\cite{ST}
\begin{enumerate}
\item
Let $A$ be a vector space with a co-multiplication $\Delta:A\rightarrow A\otimes A$. $(A,\Delta)$ is called a \textbf{Leibniz coalgebra} if
        \begin{equation}\label{eq:defi:Leibniz co}
            (\Delta\otimes\mathrm{id})\Delta(x)+(\sigma\otimes\mathrm{id})(\mathrm{id}\otimes\Delta)\Delta(x)-(\mathrm{id}\otimes\Delta)\Delta(x)=0,\;\;\forall x\in A.
        \end{equation}
        %where $\sigma$ is the twist map given by $\sigma(x\otimes y)=y\otimes x$, for all $x,y\in A.$
\item Let $(A,[-,-]_{A})$ be a Leibniz algebra and  $(A,\Delta)$ be a Leibniz coalgebra. Suppose that the following equations hold:
        \begin{small}
            \begin{equation}\label{eq:defi:Leibniz bialg1}
                \sigma(\mathcal{R}_{A}(y)\otimes\mathrm{id})\Delta(x)=(\mathcal{R}_{A}(x)\otimes\mathrm{id})\Delta(y),
            \end{equation}
            \begin{equation}\label{eq:defi:Leibniz bialg2}
                \Delta([x,y]_{A})=(\mathrm{id}\otimes\mathcal{R}_{A}(y)-(\mathcal{L}_{A}+\mathcal{R}_{A})(y)\otimes\mathrm{id})(\mathrm{id}+\sigma)\Delta(x)+(\mathrm{id}\otimes\mathcal{L}_{A}(x)+\mathcal{L}_{A}(x)\otimes\mathrm{id})\Delta(y),
            \end{equation}
        \end{small}
        for all $x,y\in A$. Such a structure is called a \textbf{Leibniz bialgebra} and is denoted by  $(A,[-,-]_{A},\Delta)$.
\end{enumerate}
    \end{defi}}

\begin{rmk}
    In fact, $(A,\Delta)$ is a Leibniz coalgebra if and only if $(A^{*},[-,-]_{A^{*}})$ is a Leibniz algebra, where  $[-,-]_{A^{*}}:A^{*}\otimes A^{*}\rightarrow A^{*}$ is the linear dual of $\Delta$, that is,
    \begin{equation}\label{linear dual}
        \langle \Delta(x),a^{*}\otimes b^{*}\rangle=\langle x,[a^{*},b^{*}]_{A^{*}}\rangle,\;\;\forall x\in A, a^{*},b^{*}\in A^{*}.
    \end{equation}
Here $\langle-,-\rangle$ is the ordinary pair between $A$ and $A^{*}$. Hence a Leibniz bialgebra $(A,[-,-]_{A}$,
$\Delta)$ is sometimes denoted by
 $((A,[-,-]_{A}),(A^{*},[-,-]_{A^{*}}))$, where the Leibniz algebra structure $(A^{*},[-,-]_{A^{*}})$ on the dual space $A^{*}$ corresponds to the Leibniz coalgebra $(A,\Delta)$ via Eq.~\eqref{linear dual}.
\end{rmk}

\begin{defi}
Let $(A,[-,-]_{A})$ be a Leibniz algebra.
Define a linear map
     $F:A\rightarrow\mathrm{End}(A\otimes A)$ by
    \begin{equation}\label{eq:F}
        F(x):=(\mathcal{L}_{A}+\mathcal{R}_{A})(x)\otimes\mathrm{id}-\mathrm{id}\otimes\mathcal{R}_{A}(x),\;\;\forall x\in A.
    \end{equation}
An element $r\in A\otimes A$ is called {\bf invariant} if $F(x)r=0$ for all $x\in A$.
%\begin{equation}\label{eq:inv}
%   ((\mathcal{L}_{A}+\mathcal{R}_{A})(x)\otimes\mathrm{id}-\mathrm{id}\otimes\mathcal{R}_{A}(x))r=0,\;\;\forall x\in A.
%\end{equation}
\end{defi}

\begin{pro}\label{pro:coalg}
    Let $(A,[-,-]_{A})$ be a Leibniz algebra  and $r=\sum\limits_{i}a_{i}\otimes b_{i}\in A\otimes A$.
   Define a co-multiplication $\Delta=\Delta_{r}:A\rightarrow A\otimes A$ by
    \begin{equation}\label{eq:cob}
        \Delta_{r}(x)=F(x)r,\;\;\forall x\in A.
    \end{equation}
     Then  $(A,\Delta_{r})$ is a Leibniz coalgebra if and only if the following equation holds:
    \begin{small}
    \begin{equation}\label{eq:coalg cond}
    \begin{split}
        &\Big(\mathrm{id}\otimes\mathrm{id}\otimes\mathcal{R}_{A}(x)-\mathrm{id}\otimes(\mathcal{L}_{A}+\mathcal{R}_{A})(x)\otimes\mathrm{id}\Big)\Big( (\sigma\otimes\mathrm{id}) [[r,r]]-\sum_{j}F(a_{j})(r-\sigma(r))\otimes b_{j}\Big)\\
        &+\Big((\mathcal{L}_{A}+\mathcal{R}_{A})(x)\otimes\mathrm{id}\otimes\mathrm{id}\Big)[[r,r]]\\
        &+\sum_{j}\Big((\mathcal{L}_{A}+\mathcal{R}_{A})(a_{j})\otimes\mathrm{id}\otimes\mathrm{id}\Big)\Big(\sigma\big(F(x)(r-\sigma(r))\big)\otimes b_{j}\Big)=0,~~\forall x\in A.
    \end{split}
    \end{equation}
\end{small}
%for all $x\in A$.
%Consequently, $(A,\Delta_{r})$ is a Leibniz coalgebra if $r$ is a semi-invariant solution of the CLYBE in $(A,[-,-]_{A})$.
\end{pro}
\begin{proof}
Since $(A,[-,-]_{A})$ is a Leibniz algebra, there is an obvious identity:
    \begin{equation}\label{eq:defi:Leibniz2}
        [[x,y]_{A},z]_{A}+[[y,x]_{A},z]_{A}=0,\;\;\forall x,y,z\in A.
    \end{equation}
For all $x\in A$, we have
\begin{eqnarray*}
    &&(\Delta_{r}\otimes\mathrm{id})\Delta_{r}(x)-(\mathrm{id}\otimes \Delta_{r})\Delta_{r}(x)+(\sigma\otimes\mathrm{id})(\mathrm{id}\otimes \Delta_{r})\Delta_{r}(x)\\
    &&=\sum_{i,j}(a_{j}\otimes[b_{j},a_{i}]_{A}\otimes [b_{i},x]_{A}-[a_{i},a_{j}]_{A}\otimes b_{j}\otimes [b_{i},x]_{A}-[a_{j},a_{i}]_{A}\otimes b_{j}\otimes [b_{i},x]_{A}\\
    &&\ \ -a_{j}\otimes [b_{j},[x,a_{i}]_{A}]_{A}\otimes b_{i}+[[x,a_{i}]_{A},a_{j}]_{A}\otimes b_{j}\otimes b_{i}
    +[a_{j},[x,a_{i}]_{A}]_{A}\otimes b_{j}\otimes b_{i}\\
    &&\ \ -a_{j}\otimes [b_{j},[a_{i},x]_{A}]_{A}\otimes b_{i}+[[a_{i},x]_{A},a_{j}]_{A}\otimes b_{j}\otimes b_{i}
    +[a_{j},[a_{i},x]_{A}]_{A}\otimes b_{j}\otimes b_{i}\\
    &&\ \ -a_{i}\otimes a_{j}\otimes[b_{j},[b_{i},x]_{A}]_{A}+a_{i}\otimes[[b_{i},x]_{A},a_{j}]_{A}\otimes b_{j}+a_{i}\otimes[a_{j},[b_{i},x]_{A}]_{A}\otimes b_{j}\\
    &&\ \ +[x,a_{i}]_{A}\otimes a_{j}\otimes [b_{j},b_{i}]_{A}-[x,a_{i}]_{A}\otimes[b_{i},a_{j}]_{A}\otimes b_{j}-[x,a_{i}]_{A}\otimes[a_{j},b_{i}]_{A}\otimes b_{j}\\
    &&\ \ +[a_{i},x]_{A}\otimes a_{j}\otimes [b_{j},b_{i}]_{A}-[a_{i},x]_{A}\otimes[b_{i},a_{j}]_{A}\otimes b_{j}-[a_{i},x]_{A}\otimes[a_{j},b_{i}]_{A}\otimes b_{j}\\
    &&\ \ +a_{j}\otimes a_{i}\otimes[b_{j},[b_{i},x]_{A}]_{A}-[[b_{i},x]_{A},a_{j}]_{A}\otimes a_{i}\otimes b_{j}-[a_{j},[b_{i},x]_{A}]_{A}\otimes a_{i}\otimes b_{j}\\
    &&\ \ -a_{j}\otimes[x,a_{i}]_{A}\otimes[b_{j},b_{i}]_{A}+[b_{i},a_{j}]_{A}\otimes[x,a_{i}]_{A}\otimes b_{j}+[a_{j},b_{i}]_{A}\otimes[x,a_{i}]_{A}\otimes b_{j}\\
    &&\ \ -a_{j}\otimes[a_{i},x]_{A}\otimes[b_{j},b_{i}]_{A}+[b_{i},a_{j}]_{A}\otimes[a_{i},x]_{A}\otimes b_{j}+[a_{j},b_{i}]_{A}\otimes[a_{i},x]_{A}\otimes b_{j})\\
    &&=\sum_{k=1}^{9}A(k),
\end{eqnarray*}
where
\begin{small}
\begin{eqnarray*}
    A(1)&=&\sum_{i,j}(a_{j}\otimes[b_{j},a_{i}]_{A}\otimes [b_{i},x]_{A}-[a_{i},a_{j}]_{A}\otimes b_{j}\otimes [b_{i},x]_{A}-[a_{j},a_{i}]_{A}\otimes b_{j}\otimes [b_{i},x]_{A}\\
    &&-a_{i}\otimes a_{j}\otimes[b_{j},[b_{i},x]_{A}]_{A}+a_{j}\otimes a_{i}\otimes[b_{j},[b_{i},x]_{A}]_{A})\\
    &\overset{\eqref{eq:defi:Leibniz}}{=}&\sum_{i,j}(a_{j}\otimes[b_{j},a_{i}]_{A}\otimes [b_{i},x]_{A}-[a_{i},a_{j}]_{A}\otimes b_{j}\otimes [b_{i},x]_{A}-[a_{j},a_{i}]_{A}\otimes b_{j}\otimes [b_{i},x]_{A}\\
    &&+a_{i}\otimes a_{j}\otimes[[b_{i},b_{j}]_{A},x]_{A})\\
    &=&(\mathrm{id}\otimes\mathrm{id}\otimes\mathcal{R}_{A}(x))
    \sum_{i,j}( a_{j}\otimes[b_{j},a_{i}]_{A}\otimes b_{i} -[a_{i},a_{j}]_{A}\otimes b_{j}\otimes  b_{i} \\
    &&
    -[a_{j},a_{i}]_{A}\otimes b_{j}\otimes b_{i}  +a_{i}\otimes a_{j}\otimes [b_{i},b_{j}]_{A} )\\
    &=&(\mathrm{id}\otimes\mathrm{id}\otimes\mathcal{R}_{A}(x))( (\sigma\otimes\mathrm{id})[[r,r]]-\sum_{j}F(a_{j})(r-\sigma(r))\otimes b_{j}  ),\\
    A(2)&=&\sum_{i,j}([[x,a_{i}]_{A},a_{j}]_{A}\otimes b_{j}\otimes b_{i}+[[a_{i},x]_{A},a_{j}]_{A}\otimes b_{j}\otimes b_{i})\overset{\eqref{eq:defi:Leibniz2}}{=}0,\\
    A(3)&=&\sum_{i,j}(-[[b_{i},x]_{A},a_{j}]_{A}\otimes a_{i}\otimes b_{j}-[a_{j},[b_{i},x]_{A}]_{A}\otimes a_{i}\otimes b_{j}),\\
%\end{eqnarray*}
%\begin{eqnarray*}
    A(4)&=&\sum_{i,j}([a_{j},[x,a_{i}]_{A}]_{A}\otimes b_{j}\otimes b_{i}+[x,a_{i}]_{A}\otimes a_{j}\otimes [b_{j},b_{i}]_{A}-[x,a_{i}]_{A}\otimes[b_{i},a_{j}]_{A}\otimes b_{j}\\
    &&-[x,a_{i}]_{A}\otimes[a_{j},b_{i}]_{A}\otimes b_{j})\\
    &=&( \mathcal{L}_{A}(x)\otimes\mathrm{id}\otimes\mathrm{id}) [[r,r]] +\sum_{i,j}(-[x,[a_{i},a_{j}]_{A}]_{A}\otimes b_{i}\otimes b_{j}+[a_{i},[x,a_{j}]_{A}]_{A}\otimes b_{i}\otimes b_{j})\\
    &\overset{\eqref{eq:defi:Leibniz}}{=}&( \mathcal{L}_{A}(x)\otimes\mathrm{id}\otimes\mathrm{id}) [[r,r]]+\sum_{i,j}[[a_{i},x]_{A},a_{j}]_{A}\otimes b_{i}\otimes b_{j},\\
    A(5)&=&\sum_{i,j}(-a_{j}\otimes[x,a_{i}]_{A}\otimes[b_{j},b_{i}]_{A}+[b_{i},a_{j}]_{A}\otimes[x,a_{i}]_{A}\otimes b_{j}+[a_{j},b_{i}]_{A}\otimes[x,a_{i}]_{A}\otimes b_{j})\\
    &=&-(\mathrm{id}\otimes\mathcal{L}_{A}(x)\otimes\mathrm{id})(\sigma\otimes\mathrm{id}) [[r,r]]+\sum_{i,j}b_{i}\otimes[x,[a_{i},a_{j}]_{A}]_{A}\otimes b_{j},
    \\
%\end{eqnarray*}
%\begin{eqnarray*}
    A(6)&=&\sum_{i,j}(-a_{j}\otimes [b_{j},[x,a_{i}]_{A}]_{A}\otimes b_{i}+a_{i}\otimes[[b_{i},x]_{A},a_{j}]_{A}\otimes b_{j})\\
    &\overset{\eqref{eq:defi:Leibniz}}{=}&-\sum_{i,j}a_{i}\otimes[x,[b_{i},a_{j}]_{A}]_{A}\otimes b_{j},
\end{eqnarray*}
\end{small}
and similarly
\begin{small}
\begin{eqnarray*}
    A(7)&=&\sum_{i,j}([a_{j},[a_{i},x]_{A}]_{A}\otimes b_{j}\otimes b_{i}+[a_{i},x]_{A}\otimes a_{j}\otimes [b_{j},b_{i}]_{A}-[a_{i},x]_{A}\otimes[b_{i},a_{j}]_{A}\otimes b_{j}\\
    &&-[a_{i},x]_{A}\otimes[a_{j},b_{i}]_{A}\otimes b_{j})\\
    &=&(\mathcal{R}_{A}(x)\otimes\mathrm{id}\otimes\mathrm{id}) [[r,r]]+\sum_{i,j} [a_{j},[a_{i},x]_{A}]_{A}\otimes b_{i}\otimes b_{j},\\
    A(8)&=&\sum_{i,j}(-a_{j}\otimes[a_{i},x]_{A}\otimes[b_{j},b_{i}]_{A}+[b_{i},a_{j}]_{A}\otimes[a_{i},x]_{A}\otimes b_{j}+[a_{j},b_{i}]_{A}\otimes[a_{i},x]_{A}\otimes b_{j})\\
    &=&-(\mathrm{id}\otimes\mathcal{R}_{A}(x)\otimes\mathrm{id})(\sigma\otimes\mathrm{id})[[r,r]]+\sum_{i,j} b_{i}\otimes [[a_{i},a_{j}]_{A},x]_{A}\otimes b_{j} ,\\
    A(9)&=&\sum_{i,j}(-a_{j}\otimes [b_{j},[a_{i},x]_{A}]_{A}\otimes b_{i}+a_{i}\otimes[a_{j},[b_{i},x]_{A}]_{A}\otimes b_{j})=-\sum_{i,j}a_{i}\otimes[[b_{i},a_{j}]_{A},x]_{A}\otimes b_{j}.
\end{eqnarray*}
\end{small}
Note that
\begin{small}
\begin{eqnarray*}
    &&\sum_{i,j}([[a_{i},x]_{A},a_{j}]_{A}\otimes b_{i}\otimes b_{j}-[[b_{i},x]_{A},a_{j}]_{A}\otimes a_{i}\otimes b_{j})\\
    &&=\sum_{j}(\mathcal{R}_{A}(a_{j})\otimes\mathrm{id}\otimes\mathrm{id})(\sigma(F(x)(r-\sigma(r)))\otimes b_{j})+\sum_{i,j}(-[b_{i},a_{j}]_{A}\otimes[x,a_{i}]_{A}\otimes b_{j}\\
    &&\ \ -[b_{i},a_{j}]_{A}\otimes[a_{i},x]_{A}\otimes b_{j}+[a_{i},a_{j}]_{A}\otimes[x,b_{i}]_{A}\otimes b_{j}+[a_{i},a_{j}]_{A}\otimes[b_{i},x]_{A}\otimes b_{j}),\\
    &&\sum_{i,j}([a_{j},[a_{i},x]_{A}]_{A}\otimes b_{i}\otimes b_{j}-[a_{j},[b_{i},x]_{A}]_{A}\otimes a_{i}\otimes b_{j})\\
    &&=\sum_{j}(\mathcal{L}_{A}(a_{j})\otimes\mathrm{id}\otimes\mathrm{id})(\sigma(F(x)(r-\sigma(r)))\otimes b_{j})+\sum_{i,j}(-[a_{j},b_{i}]_{A}\otimes[x,a_{i}]_{A}\otimes b_{j}\\
    &&\ \ -[a_{j},b_{i}]_{A}\otimes[a_{i},x]_{A}\otimes b_{j}+[a_{j},a_{i}]_{A}\otimes[x,b_{i}]_{A}\otimes b_{j}+[a_{j},a_{i}]_{A}\otimes[b_{i},x]_{A}\otimes b_{j}).
\end{eqnarray*}
\end{small}
Then we have
\begin{small}
\begin{eqnarray*}
    &&\sum_{k=1}^{9}A(k)
    -(\mathrm{id}\otimes\mathrm{id}\otimes\mathcal{R}_{A}(x))((\sigma\otimes\mathrm{id})[[r,r]]-\sum_{j}F(a_{j})(r-\sigma(r))\otimes b_{j})
    \\
    &&-\sum_{j}((\mathcal{L}_{A}+\mathcal{R}_{A})(a_{j})\otimes\mathrm{id}\otimes\mathrm{id})(\sigma(F(x)(r-\sigma(r)))\otimes b_{j})\\
    &&
   +(\mathrm{id}\otimes(\mathcal{L}_{A}+\mathcal{R}_{A})(x)\otimes\mathrm{id})(\sigma\otimes\mathrm{id}) [[r,r]]-((\mathcal{L}_{A}+\mathcal{R}_{A})(x)\otimes\mathrm{id}\otimes\mathrm{id})[[r,r]]
    \\
    &&=\sum_{i,j}(-[b_{i},a_{j}]_{A}\otimes[x,a_{i}]_{A}\otimes b_{j}-[b_{i},a_{j}]_{A}\otimes[a_{i},x]_{A}\otimes b_{j}+[a_{i},a_{j}]_{A}\otimes[x,b_{i}]_{A}\otimes b_{j}\\
    &&\ \ +[a_{i},a_{j}]_{A}\otimes[b_{i},x]_{A}\otimes b_{j}-[a_{j},b_{i}]_{A}\otimes[x,a_{i}]_{A}\otimes b_{j}-[a_{j},b_{i}]_{A}\otimes[a_{i},x]_{A}\otimes b_{j}\\
    &&\ \ +[a_{j},a_{i}]_{A}\otimes[x,b_{i}]_{A}\otimes b_{j}+[a_{j},a_{i}]_{A}\otimes[b_{i},x]_{A}\otimes b_{j}-a_{i}\otimes[x,[b_{i},a_{j}]_{A}]_{A}\otimes b_{j}\\
    &&\ \ -a_{i}\otimes[[b_{i},a_{j}]_{A},x]_{A}\otimes b_{j}+b_{i}\otimes[x,[a_{i},a_{j}]_{A}]_{A}\otimes b_{j}+b_{i}\otimes[[a_{i},a_{j}]_{A},x]_{A}\otimes b_{j})\\
    &&=(\mathrm{id}\otimes(\mathcal{L}_{A}+\mathcal{R}_{A})(x)\otimes\mathrm{id})\sum_{j}F(a_{j})(r-\sigma(r))\otimes b_{j}.
\end{eqnarray*}
\end{small}
Hence the conclusion follows.
\end{proof}

    %Now we introduce the notion of Leibniz bialgebras.

\begin{pro}\label{pro:matched pair}
Let $(A,[-,-]_{A})$ be a Leibniz algebra  and $r=\sum\limits_{i}a_{i}\otimes b_{i}\in A\otimes A$. Define a co-multiplication $\Delta_{r}:A\rightarrow A\otimes A$ by Eq.~\eqref{eq:cob}.
\begin{enumerate}
    \item \label{it:1} Eq.~\eqref{eq:defi:Leibniz bialg1} holds if and only if
    \begin{equation}\label{mp1}
        (\mathcal{R}_{A}(y)\otimes\mathrm{id})F(x)(r-\sigma(r))=0,\;\;\forall x,y\in A.
    \end{equation}
        \item \label{it:2} Eq.~\eqref{eq:defi:Leibniz bialg2} holds if and only if
        \begin{small}
        \begin{equation}\label{mp2}
            \begin{split}
        &(\mathcal{L}_{A}(y)\otimes\mathrm{id})\sigma(F(x)(r-\sigma(r)))+(\mathcal{R}_{A}(y)\otimes\mathrm{id})(\sigma-\mathrm{id})(F(x)(r-\sigma(r)))\\
        &-\sigma(\mathcal{R}_{A}(y)\otimes\mathrm{id})F(x)(r-\sigma(r))-\sigma(\mathcal{R}_{A}(x)\otimes\mathrm{id})F(y)(r-\sigma(r))=0,\;\;\forall x,y\in A.
        \end{split}
        \end{equation}
    \end{small}
\end{enumerate}
\end{pro}
\begin{proof}  (\ref{it:1}).~ By Eq.~\eqref{eq:cob}, we have
        \begin{eqnarray*}
            &&\sigma(\mathcal{R}_{A}(y)\otimes\mathrm{id})\Delta_{r}(x)-(\mathcal{R}_{A}(x)\otimes\mathrm{id})\Delta_{r}(y)\\
            &&=\sum_{i}(-[b_{i},x]_{A}\otimes[a_{i},y]_{A}+b_{i}\otimes[[x,a_{i}]_{A},y]_{A}+b_{i}\otimes[[a_{i},x]_{A},y]_{A}\\
            &&\ \ +[a_{i},x]_{A}\otimes[b_{i},y]_{A}-[[y,a_{i}]_{A},x]_{A}\otimes b_{i}-[[a_{i},y]_{A},x]_{A}\otimes b_{i})\\
            &&\overset{\eqref{eq:defi:Leibniz2}}{=}\sum_{i}(-[b_{i},x]_{A}\otimes[a_{i},y]_{A}+[a_{i},x]_{A}\otimes[b_{i},y]_{A})\\
            &&=\sigma(\mathcal{R}_{A}(y)\otimes\mathrm{id})F(x)(r-\sigma(r)).
        \end{eqnarray*}
        Hence the conclusion follows.

  (\ref{it:2}).       By Eq.~\eqref{eq:cob}, we have
        \begin{eqnarray*}
            &&\Delta_{r}([x,y]_{A})-(\mathrm{id}\otimes\mathcal{L}_{A}(x)+\mathcal{L}_{A}(x)\otimes\mathrm{id})\Delta_{r}(y)\\
            &&-(\mathrm{id}\otimes\mathcal{R}_{A}(y)-(\mathcal{L}_{A}+\mathcal{R}_{A})(y)\otimes\mathrm{id})(\mathrm{id}+\sigma)\Delta_{r}(x)\\
            &&=\sum_{i}(-a_{i}\otimes[b_{i},[x,y]_{A}]_{A}+[[x,y]_{A},a_{i}]_{A}\otimes b_{i}+[a_{i},[x,y]_{A}]_{A}\otimes b_{i}\\
            &&\ \ +a_{i}\otimes[x,[b_{i},y]_{A}]_{A}-[y,a_{i}]_{A}\otimes[x,b_{i}]_{A}-[a_{i},y]_{A}\otimes[x,b_{i}]_{A}\\
            &&\ \ +[x,a_{i}]_{A}\otimes[b_{i},y]_{A}-[x,[y,a_{i}]_{A}]_{A}\otimes b_{i}-[x,[a_{i},y]_{A}]_{A}\otimes b_{i}\\
            &&\ \ +a_{i}\otimes[[b_{i},x]_{A},y]_{A}-[x,a_{i}]_{A}\otimes[b_{i},y]_{A}-[a_{i},x]_{A}\otimes [b_{i},y]_{A}\\
            &&\ \ +[b_{i},x]_{A}\otimes[a_{i},y]_{A}-b_{i}\otimes[[x,a_{i}]_{A},y]_{A}-b_{i}\otimes[[a_{i},x]_{A},y]_{A}\\
            &&\ \ -[y,a_{i}]_{A}\otimes[b_{i},x]_{A}+[y,[x,a_{i}]_{A}]_{A}\otimes b_{i}+[y,[a_{i},x]_{A}]_{A}\otimes b_{i}\\
            &&\ \ -[y,[b_{i},x]_{A}]_{A}\otimes a_{i}+[y,b_{i}]_{A}\otimes[x,a_{i}]_{A}+[y,b_{i}]_{A}\otimes[a_{i},x]_{A}\\
            &&\ \ -[a_{i},y]_{A}\otimes[b_{i},x]_{A}+[[x,a_{i}]_{A},y]_{A}\otimes b_{i}+[[a_{i},x]_{A},y]_{A}\otimes b_{i}\\
            &&\ \ -[[b_{i},x]_{A},y]_{A}\otimes a_{i}+[b_{i},y]_{A}\otimes[x,a_{i}]_{A}+[b_{i},y]_{A}\otimes [a_{i},x]_{A})\\
            &&=\sum_{k=1}^{9}B(k),
        \end{eqnarray*}
        where
        \begin{small}
        \begin{eqnarray*}
            B(1)&=&\sum_{i}([x,a_{i}]_{A}\otimes[b_{i},y]_{A}-[x,a_{i}]_{A}\otimes [b_{i},y]_{A})=0,\\
            B(2)&=&\sum_{i}(-a_{i}\otimes[b_{i},[x,y]_{A}]_{A}+a_{i}\otimes[x,[b_{i},y]_{A}]_{A}+a_{i}\otimes[[b_{i},x]_{A},y]_{A})\overset{\eqref{eq:defi:Leibniz}}{=}0,\\
            B(3)&=&\sum_{i}([[x,y]_{A},a_{i}]_{A}\otimes b_{i}-[x,[y,a_{i}]_{A}]_{A}\otimes b_{i}+[y,[x,a_{i}]_{A}]_{A}\otimes b_{i})\overset{\eqref{eq:defi:Leibniz}}{=}0,\\
            B(4)&=&\sum_{i}(-b_{i}\otimes[[x,a_{i}]_{A},y]_{A}-b_{i}\otimes[[a_{i},x]_{A},y]_{A})\overset{\eqref{eq:defi:Leibniz2}}{=}0,\\
            B(5)&=&\sum_{i}([[x,a_{i}]_{A},y]_{A}\otimes b_{i}+[[a_{i},x]_{A},y]_{A}\otimes b_{i})\overset{\eqref{eq:defi:Leibniz2}}{=}0,\\
            B(6)&=&\sum_{i}(-[a_{i},x]_{A}\otimes[b_{i},y]_{A}+[b_{i},x]_{A}\otimes[a_{i},y]_{A})=-\sigma(\mathcal{R}_{A}(y)\otimes\mathrm{id})F(x)(r-\sigma(r)),\\
            B(7)&=&\sum_{i}(-[a_{i},y]_{A}\otimes[b_{i},x]_{A}+[b_{i},y]_{A}\otimes[a_{i},x]_{A})=-\sigma(\mathcal{R}_{A}(x)\otimes\mathrm{id})F(y)(r-\sigma(r)),\\
            B(8)&=&\sum_{i}(-[y,a_{i}]_{A}\otimes [x,b_{i}]_{A}+[y,b_{i}]_{A}\otimes[x,a_{i}]_{A}
            -[y,a_{i}]_{A}\otimes[b_{i},x]_{A}+[y,b_{i}]_{A}\otimes[a_{i},x]_{A}\\
            &&+[y,[a_{i},x]_{A}]_{A}\otimes b_{i}-[y,[b_{i},x]_{A}]_{A}\otimes a_{i})\\
            &=&(\mathcal{L}_{A}(y)\otimes\mathrm{id})\sigma(F(x)(r-\sigma(r))),\\
            B(9)&=&\sum_{i}(-[a_{i},y]_{A}\otimes[x,b_{i}]_{A}+[b_{i},y]_{A}\otimes[x,a_{i}]_{A}+[a_{i},[x,y]_{A}]_{A}\otimes b_{i}-[x,[a_{i},y]_{A}]_{A}\otimes b_{i}\\
            &&-[[b_{i},x]_{A},y]_{A}\otimes a_{i})\\
            &\overset{\eqref{eq:defi:Leibniz}}{=}&\sum_{i}(-[a_{i},y]_{A}\otimes[x,b_{i}]_{A}+[b_{i},y]_{A}\otimes[x,a_{i}]_{A}-[[x,a_{i}]_{A},y]_{A}\otimes b_{i}+[[x,b_{i}]_{A},y]_{A}\otimes a_{i})\\
            &=&(\mathcal{R}_{A}(y)\otimes\mathrm{id})(\sigma-\mathrm{id})(F(x)(r-\sigma(r))).
        \end{eqnarray*}
    \end{small}
The proof is finished.
\end{proof}

Combining Propositions \ref{pro:coalg} and \ref{pro:matched pair}, we have the following corollary.

\begin{cor}\label{cor:quasitr}
    Let $(A,[-,-]_{A})$ be a Leibniz algebra and $r=\sum\limits_{i}a_{i}\otimes b_{i}\in A\otimes A$. Let $\Delta_{r}:A\rightarrow A\otimes A$ be a co-multiplication given by Eq.~\eqref{eq:cob}. Then $(A,[-,-]_{A},\Delta_{r})$ is a Leibniz bialgebra if and only if Eq.~\eqref{eq:coalg cond}, Eq.~\eqref{mp1} and Eq.~\eqref{mp2} hold.  In particular,  if $r$ is a  solution of the CLYBE in $(A,[-,-]_{A})$
 and the skew-symmetric part $\frac{1}{2}(r-\sigma(r))$ is invariant, that is,
 \begin{equation}\label{s.i.}
     F(x)(r-\sigma(r))=0,\;\forall x\in A,
 \end{equation}
 then $(A,[-,-]_{A},\Delta_{r})$ is a  Leibniz bialgebra.
\end{cor}

\begin{defi}
    Let $(A,[-,-]_{A})$ be a Leibniz algebra and $r\in A\otimes A$.
    If $r$ is a solution of the CLYBE in $(A,[-,-]_{A})$ satisfying   Eq.~\eqref{s.i.},
    then we call that  the Leibniz bialgebra $(A,[-,-]_{A},\Delta_{r})$
   with $\Delta_r$ defined  by Eq.~\eqref{eq:cob} is a \textbf{quasi-triangular Leibniz bialgebra}.
\end{defi}

Note  that {\bf triangular Leibniz bialgebras} \cite{ST}   are obtained from symmetric solutions of the CLYBE.
Hence quasi-triangular Leibniz bialgebras contain triangular ones as a subclass.

\begin{pro}\label{pro:2.12}
If $(A,[-,-]_{A},\Delta_{r})$ is a quasi-triangular Leibniz bialgebra, then $(A$,
$[-,-]_{A}$,
$\Delta_{\sigma(r)})$ is also a quasi-triangular Leibniz bialgebra.
\end{pro}
\begin{proof}
    It follows from Proposition \ref{pro:equation equivalence}.
\end{proof}

\begin{ex}\label{ex:2.13}
Consider the 4-dimensional Leibniz algebra $(A,[-,-]_{A})$ defined with respect to a basis $\{ e_{1},e_{2},e_{3},e_{4} \}$ by the following nonzero products:
\begin{eqnarray*}
    [e_{1},e_{2}]_{A}=e_{1}, \; [e_{2},e_{1}]_{A}=-e_{1},\; [e_{1},e_{3}]_{A}=-e_{4},\; [e_{2},e_{3}]_{A}=e_{3}.
\end{eqnarray*}
Then it is straightforward to check
\begin{equation*}
    r=e_{3}\otimes e_{1}+e_{4}\otimes e_{2}
\end{equation*}
is a   solution of the CLYBE in $(A,[-,-]_{A})$ and satisfies Eq.~\eqref{s.i.},
and thus gives rise to a quasi-triangular Leibniz bialgebra $(A,[-,-]_{A},\Delta_{r})$ with $\Delta_r$ defined by Eq.~\eqref{eq:cob}.
\end{ex}

%At the end of this section, we show that quasi-triangular Leibniz bialgebras are accompanied with Leibniz algebra homomorphisms.

Let $A$ and $V$ be vector spaces.
We identify $r\in A\otimes A$ as a linear map $T_{r}:A^{*}\rightarrow A$ via
\begin{equation}\label{eq:id}
    \langle T_{r}(a^{*}),b^{*}\rangle=\langle r, a^{*}\otimes b^{*}\rangle,\;\;\forall a^{*},b^{*}\in A^{*}.
\end{equation}
For a linear map $f:A\rightarrow\mathrm{End}(V)$, define a linear map $f^{*}:A\rightarrow\mathrm{End}(V^{*})$ by
\begin{equation}
    \langle f^{*}(x)u^{*}, v\rangle=-\langle u^{*},f(x)v\rangle,\;\;\forall x\in A, u^{*}\in V^{*},v\in V {{.}}
\end{equation}
%where $\langle-,-\rangle$ is the ordinary pair between $V$ and $V^{*}$.

%Now we study the operator forms of the CLYBE.

\begin{pro}\label{pro:CLYBE1}
Let $(A,[-,-]_{A})$ be a Leibniz algebra and $r\in A\otimes A$.
Then $r$ is a solution of the CLYBE in $(A,[-,-]_{A})$ if and only if the following equation holds:
\begin{equation}\label{eq:T homo2}
[T_{\sigma(r)}(a^{*}),T_{\sigma(r)}(b^{*})]_{A}=T_{\sigma(r)}(\mathcal{L}^{*}_{A}(T_{r}(a^{*}))b^{*}-(\mathcal{L}^{*}_{A}+\mathcal{R}^{*}_{A})(T_{\sigma(r)}(b^{*}))a^{*}), \;\;\forall a^{*},b^{*}\in A^{*}.
\end{equation}
If in addition $r$ satisfies   Eq.~\eqref{s.i.}, then $r$ is a solution of the CLYBE in $(A,[-,-]_{A})$ if and only if
the following equation holds:
\begin{equation}\label{eq:T homo1}
    [T_{r}(a^{*}),T_{r}(b^{*})]_{A}=T_{r}(\mathcal{L}^{*}_{A}(T_{r}(a^{*}))b^{*}-(\mathcal{L}^{*}_{A}+\mathcal{R}^{*}_{A})(T_{\sigma(r)}(b^{*}))a^{*}).
\end{equation}
\end{pro}
\begin{proof}
We set that $r=\sum\limits_{i} a_{i}\otimes b_{i}\in A\otimes A$. For any $a^{*},b^{*},c^{*}\in A^{*}$,  we have
{\small
\begin{eqnarray*}
\langle [T_{\sigma(r)}(a^{*}),T_{\sigma(r)}(b^{*})]_{A},c^{*}\rangle&=&-\langle T_{\sigma(r)}(a^{*}),\mathcal{R}^{*}_{A}( T_{\sigma(r)}(b^{*}) )c^{*}\rangle=-\langle r, \mathcal{R}^{*}_{A}( T_{\sigma(r)}(b^{*}) )c^{*}\otimes a^{*}\rangle\\
&=&-\sum_{i}\langle a_{i}, \mathcal{R}^{*}_{A}( T_{\sigma(r)}(b^{*}) )c^{*}\rangle\langle b_{i}, a^{*}\rangle=
\sum_{i}\langle [a_{i},T_{\sigma(r)}(b^{*})]_{A},c^{*}\rangle\langle b_{i},a^{*}\rangle\\
&=&-\sum_{i}\langle T_{\sigma(r)}(b^{*}),\mathcal{L}^{*}_{A}(a_{i})c^{*}\rangle\langle b_{i},a^{*}\rangle=-\sum_{i}\langle r, \mathcal{L}^{*}_{A}(a_{i})c^{*}\otimes b^{*}\rangle\langle b_{i},a^{*}\rangle\\
&=&-\sum_{i,j}\langle a_{j},\mathcal{L}^{*}_{A}(a_{i})c^{*}\rangle\langle b_{j},b^{*}\rangle\langle b_{i},a^{*}\rangle\\
&=&\sum_{i,j}\langle b_{i}\otimes b_{j}\otimes [a_{i},a_{j}]_{A}, a^{*}\otimes b^{*}\otimes c^{*}\rangle,\\
\langle T_{\sigma(r)}( \mathcal{L}^{*}_{A}(T_{r}(a^{*}))b^{*} ), c^{*}\rangle&=&\langle r, c^{*}\otimes \mathcal{L}^{*}_{A}( T_{r}(a^{*}) ) b^{*}\rangle=\sum_{i}\langle a_{i},c^{*}\rangle\langle b_{i},\mathcal{L}^{*}_{A}( T_{r}(a^{*} ))b^{*}\rangle\\
&=&-\sum_{i}\langle a_{i}, c^{*}\rangle\langle [T_{r}(a^{*}),b_{i}]_{A}, b^{*}\rangle=\sum_{i}\langle a_{i},c^{*}\rangle\langle T_{r}(a^{*}),\mathcal{R}^{*}_{A}(b_{i})b^{*}\rangle\\
&=&\sum_{i}\langle a_{i},c^{*}\rangle\langle r,a^{*}\otimes\mathcal{R}^{*}_{A}(b_{i})b^{*}\rangle=\sum_{i,j}\langle a_{i},c^{*}\rangle\langle a_{j},a^{*}\rangle\langle b_{j},\mathcal{R}^{*}_{A}(b_{i})b^{*}\rangle\\
&=&-\sum_{i,j}\langle a_{j}\otimes [b_{j},b_{i}]_{A}\otimes a_{i},a^{*}\otimes b^{*}\otimes c^{*} \rangle,
\end{eqnarray*} }
and similarly
\begin{equation*}
    -\langle T_{\sigma(r)}( (\mathcal{L}^{*}_{A}+\mathcal{R}^{*}_{A})( T_{\sigma(r)}(b^{*} )a^{*} ) ), c^{*}\rangle=\sum_{i,j} \langle ([a_{j},b_{i}]_{A}+[b_{i},a_{j}]_{A})\otimes b_{j}\otimes a_{i}, a^{*}\otimes b^{*}\otimes c^{*}\rangle.
\end{equation*}
Hence we have
\begin{eqnarray*}
&&\langle [T_{\sigma(r)}(a^{*}),T_{\sigma(r)}(b^{*})]_{A}-T_{\sigma(r)}(\mathcal{L}^{*}_{A}(T_{r}(a^{*}))b^{*}-(\mathcal{L}^{*}_{A}+\mathcal{R}^{*}_{A})(T_{\sigma(r)}(b^{*}))a^{*}), c^{*}\rangle\\
&&=\langle b_{i}\otimes b_{j}\otimes[a_{i},a_{j}]_{A}+a_{j}\otimes[b_{j},b_{i}]_{A}\otimes a_{i}-([a_{j},b_{i}]_{A}+[b_{i},a_{j}]_{A})\otimes b_{j}\otimes a_{i},a^{*}\otimes b^{*}\otimes c^{*}\rangle\\
&&
=\langle \xi[[r,r]], a^{*}\otimes b^{*}\otimes c^{*}\rangle,
\end{eqnarray*}
where $\xi\in\mathrm{End}(A\otimes A\otimes A)$ is given by $\xi(x\otimes y\otimes z)=y\otimes z\otimes x$ for all $x,y,z\in A$.
That is, $r$ is a solution of the CLYBE in $(A,[-,-]_{A})$ if and only if \eqref{eq:T homo2} holds.
Similarly, we have
{\small
\begin{eqnarray*}
    \langle [T_{r}(a^{*}),T_{r}(b^{*})]_{A},c^{*}\rangle&=&\sum_{i,j}\langle a_{i}\otimes a_{j}\otimes[b_{i},b_{j}]_{A},a^{*}\otimes b^{*}\otimes c^{*}\rangle,\\
    \langle T_{r}(\mathcal{L}_{A}^{*}(T_{r}(a^{*}))b^{*}), c^{*}\rangle&=&-\sum_{i,j}\langle a_{j}\otimes[b_{j},a_{i}]_{A}\otimes b_{i},a^{*}\otimes b^{*}\otimes c^{*}\rangle,\\
    -\langle T_{r} ( ( \mathcal{L}_{A}^{*} +\mathcal{R}_{A}^{*} ) (T_{\sigma(r)}(b^{*}))a^{*}  ),c^{*}\rangle&=&\sum_{i,j}\langle ([a_{i},a_{j}]_{A}+[a_{j},a_{i}]_{A})\otimes b_{j}\otimes b_{i}, a^{*}\otimes b^{*}\otimes c^{*}\rangle.
\end{eqnarray*}
}
Hence we obtain that
{\small
\begin{eqnarray*}
    &&\langle [T_{r}(a^{*}),T_{r}(b^{*})]_{A},c^{*}- T_{r}(\mathcal{L}^{*}_{A}(T_{r}(a^{*}))b^{*}-(\mathcal{L}^{*}_{A}+\mathcal{R}^{*}_{A})(T_{\sigma(r)}(b^{*}))a^{*}) \rangle\\
    &&=\sum_{i,j}\langle a_{i}\otimes a_{j}\otimes[b_{i},b_{j}]_{A}+a_{j}\otimes[b_{j},a_{i}]_{A}\otimes b_{i}-([a_{i},a_{j}]_{A}+[a_{j},a_{i}]_{A})\otimes b_{j}\otimes b_{i}, a^{*}\otimes b^{*}\otimes c^{*}\rangle\\
    &&=\langle (\sigma\otimes\mathrm{id})[[r,r]]-\sum_{j} F(a_{j})(r-\sigma(r))\otimes b_{j}, a^{*}\otimes b^{*}\otimes c^{*}\rangle.
\end{eqnarray*}}
Thus we deduce that if  $r$ satisfies Eq.~\eqref{s.i.}, then $r$ is a solution of the CLYBE in $(A,[-,-]_{A})$ if and only if
Eq.~\eqref{eq:T homo1} holds.
\end{proof}

\begin{lem}\label{lem:dual mult}
Let  $(A,[-,-]_{A})$  be a Leibniz algebra and $r\in A\otimes A$.
Define a co-multiplication $\Delta_{r}:A\rightarrow A\otimes A$ by Eq.~\eqref{eq:cob}, whose linear dual is denoted by $[-,-]_{r}:A^{*}\otimes A^{*}\rightarrow A^{*}$.
Then the multiplication $[-,-]_{r}$ on $A^{*}$ is given by
\begin{equation}\label{eq:dual mul}
[a^{*},b^{*}]_{r}=\mathcal{L}^{*}_{A}(T_{r}(a^{*}))b^{*}-(\mathcal{L}^{*}_{A}+\mathcal{R}^{*}_{A})(T_{\sigma(r)}(b^{*}))a^{*},\;\;\forall a^{*},b^{*}\in A^{*}.
\end{equation}
\end{lem}
\begin{proof}
For any $x\in A, a^{*},b^{*}\in A^{*}$,  we have
    \begin{eqnarray*}
        \langle x,[a^{*},b^{*}]_{r}\rangle&=&\langle \Delta_{r}(x),a^{*}\otimes b^{*}\rangle\\
        &\overset{\eqref{eq:cob}}{=}&\langle ((\mathcal{L}_{A}+\mathcal{R}_{A})(x)\otimes\mathrm{id}-\mathrm{id}\otimes\mathcal{R}_{A}(x))r, a^{*}\otimes b^{*}\rangle\\
        &=&\langle r,(-(\mathcal{L}^{*}_{A}+\mathcal{R}^{*}_{A})(x)\otimes\mathrm{id}+\mathrm{id}\otimes\mathcal{R}^{*}_{A}(x))a^{*}\otimes b^{*}\rangle\\
        &=&\langle r,a^{*}\otimes \mathcal{R}^{*}_{A}(x)b^{*}\rangle-\langle \sigma(r),b^{*}\otimes ( \mathcal{L}^{*}_{A}+ \mathcal{R}^{*}_{A} )(x)a^{*}\rangle\\
        &=&\langle T_{r}(a^{*}),\mathcal{R}^{*}_{A}(x)b^{*}\rangle-\langle
        T_{\sigma(r)}(b^{*}), ( \mathcal{L}^{*}_{A}+ \mathcal{R}^{*}_{A} )(x)a^{*}\rangle\\
        &=&-\langle [T_{r}(a^{*}),x]_{A},b^{*}\rangle+\langle [T_{\sigma(r)}(b^{*}),x]_{A}+[x,T_{\sigma(r)}(b^{*})]_{A},a^{*}\rangle\\
        &=&\langle x,\mathcal{L}^{*}_{A}(T_{r}(a^{*}))b^{*}-(\mathcal{L}^{*}_{A}+\mathcal{R}^{*}_{A})(T_{\sigma(r)}(b^{*}))a^{*}\rangle.
    \end{eqnarray*}
Thus Eq.~\eqref{eq:dual mul} holds.
\end{proof}

%At the end of this section, we give a significant property of quasi-triangular Leibniz bialgebras.

\begin{cor}\label{cor:homo}
Let $(A,[-,-]_{A},\Delta_{r})$ be a quasi-triangular Leibniz bialgebra.
Then $T_{r}$ and $T_{\sigma(r)}$ are both Leibniz algebra homomorphisms from $(A^{*},[-,-]_{r})$ to $(A,[-,-]_{A})$, where the multiplication $[-,-]_{r}$ is the linear dual of $\Delta_{r}$.
\end{cor}
\begin{proof}
By {{the}} assumption, $r$ is a solution of the CLYBE in
$(A,[-,-]_{A})$ satisfying  Eq.~\eqref{s.i.} and
$(A^{*},[-,-]_{r})$ is a Leibniz algebra. Then by Proposition
\ref{pro:CLYBE1} and Lemma \ref{lem:dual mult}, we have
\begin{equation*}
[T_{r}(a^{*}),T_{r}(b^{*})]_{A}=T_{r}([a^{*},b^{*}]_{r})
,\;
[T_{\sigma(r)}(a^{*}),T_{\sigma(r)}(b^{*})]_{A}=T_{\sigma(r)}([a^{*},b^{*}]_{r})
,\;\forall a^{*},b^{*}\in A^{*},
\end{equation*}
that is, $T_{r}$ and $T_{\sigma(r)}$ are both Leibniz algebra homomorphisms. %from $(A,[-,-]_{A})$ to $(A^{*},[-,-]_{r})$.
\end{proof}

\section{Factorizable Leibniz bialgebras}\label{sec3}

In this section, we study factorizable Leibniz bialgebras as a subclass of quasi-triangular Leibniz bialgebras, which lead to a factorization of the underlying Leibniz algebras.
We show that  there is a canonical factorizable Leibniz bialgebra structure on the  double space of an arbitrary Leibniz bialgebra.

Let $(A,[-,-]_{A},\Delta_{r})$ be a quasi-triangular Leibniz bialgebra.
Recall that the operator $T_{r-\sigma(r)}$ is defined by
\begin{equation}\label{factorizable operator}
T_{r-\sigma(r)}=T_{r}-T_{\sigma(r)}:A^{*}\rightarrow A.
\end{equation}
If $T_{r-\sigma(r)}=0$, then $r$ is symmetric and  $(A,[-,-]_{A},\Delta_{r})$ is triangular.
Factorizable Leibniz bialgebras are however concerned with the opposite case that $T_{r-\sigma(r)}$ is nondegenerate.

\begin{defi}
    A quasi-triangular Leibniz bialgebra $(A,[-,-]_{A},\Delta_{r})$ is called \textbf{factorizable} if $T_{r-\sigma(r)}:A^{*}\rightarrow A$ defined by Eq.~\eqref{factorizable operator} is a linear isomorphism of vector spaces.
\end{defi}

As a subclass of quasi-triangular Leibniz bialgebras, factorizable Leibniz bialgebras also appear in pairs.

\begin{cor}\label{cor:fact pair}
    If  $(A,[-,-]_{A},\Delta_{r})$  is a factorizable Leibniz bialgebra, then  $(A,[-,-]_{A}$,
    $\Delta_{\sigma(r)})$  is also a factorizable Leibniz bialgebra.
\end{cor}
\begin{proof}
    If $T_{r-\sigma(r)}$ is a linear isomorphism, then $$T_{\sigma(r)-\sigma(\sigma(r))}=T_{\sigma(r)-r}=-T_{r-\sigma(r)}$$ is also a linear isomorphism. Hence the conclusion follows from Proposition \ref{pro:2.12}.
\end{proof}

Let $(A,[-,-]_{A})$ be a Leibniz algebra and $r\in A\otimes A$. Define a linear map $T_{r}\oplus T_{\sigma(r)}:A^{*}\rightarrow A\oplus A$ by
\begin{equation*}
    T_{r}\oplus T_{\sigma(r)}(a^{*})=(T_{r}(a^{*}),T_{\sigma(r)}(a^{*})),\;\forall a^{*}\in A^{*}.
\end{equation*}
Then we have the following lemma which justifies the terminology of factorizable Leibniz bialgebras.

\begin{lem}\label{lem:fact}
    Let $(A,[-,-]_{A},\Delta_{r})$ be a factorizable Leibniz bialgebra. Then $\mathrm{Im}(T_{r}\oplus T_{\sigma(r)})$ is a Leibniz subalgebra of the direct sum  $A\oplus A$ of Leibniz algebras and is isomorphic to $(A^{*},[-,-]_{r})$. Moreover, for any $x\in A$, there is a unique decomposition
    $x=x_{1}-x_{2},$
    where $(x_{1},x_{2})\in\mathrm{Im}(T_{r}\oplus T_{\sigma(r)}).$
\end{lem}
\begin{proof}
    By Corollary \ref{cor:homo}, $T_{r}\oplus T_{\sigma(r)}$ is a Leibniz algebra homomorphism from $(A^{*},[-,-]_{r})$ to the direct sum $A\oplus A$ of  Leibniz algebras.
    By the nondegeneracy of $T_{r-\sigma(r)}:A^{*}\rightarrow A$,
    $\mathrm{Ker}(T_{r}\oplus T_{\sigma(r)})=0$, and thus
    $\mathrm{Im}(T_{r}\oplus T_{\sigma(r)})$ is isomorphic to $(A^{*},[-,-]_{r})$ as Leibniz algebras.  Moreover, %since $T_{r-\sigma(r)}$ is nondegenerate,
    any element $x\in A$
    can be uniquely expressed as
    \begin{equation*}
        x=T_{r-\sigma(r)}T_{r-\sigma(r)}^{-1}(x)=T_{r}T_{r-\sigma(r)}^{-1}(x)-T_{\sigma(r)}T_{r-\sigma(r)}^{-1}(x)=x_{1}-x_{2},
    \end{equation*}
    where $x_{1}=T_{r}(T_{r-\sigma(r)}^{-1}(x))\in\mathrm{Im}(T_{r}),\; x_{2}=T_{\sigma(r)}(T_{r-\sigma(r)}^{-1}(x))\in\mathrm{Im}(T_{\sigma(r)})$. The proof is finished.
\end{proof}

\delete{
\begin{rmk}
    The factorizable cases were first considered in the context of Lie bialgebras, which are related to certain factorization problems in integrable systems (see \cite{Kos,Res,Sem,Lan} for origins and further study).

\end{rmk}

\cm{The above remark is necessary?}}

\begin{ex}
Let $(A,[-,-]_{A},\Delta_{r})$ be the quasi-triangular Leibniz bialgebra given by Example \ref{ex:2.13}.
Then we have
\begin{equation*}
    r-\sigma(r)=-e_{1}\otimes e_{3}-e_{2}\otimes e_{4}+e_{3}\otimes e_{1}+e_{4}\otimes e_{2},
\end{equation*}
which gives
\begin{equation*}
    T_{r-\sigma(r)}(e^{*}_{1})=-e_{3},\;
    T_{r-\sigma(r)}(e^{*}_{2})=-e_{4},\;
    T_{r-\sigma(r)}(e^{*}_{3})=e_{1},\;
    T_{r-\sigma(r)}(e^{*}_{4})=e_{2}.
\end{equation*}
Hence $T_{r-\sigma(r)}$  is nondegenerate and thus $(A,[-,-]_{A},\Delta_{r})$ is factorizable Leibniz bialgebra.
\end{ex}

\delete{
Recall the representation theory of Leibniz algebras.

\begin{defi}
A \textbf{representation} of a Leibniz algebra $(A,[-,-]_{A})$ is a triple $(\mathcal{L},\mathcal{R},V)$, where $V$ is a vector space and $\mathcal{L},\mathcal{R}:A\rightarrow\mathrm{End}(V)$ are linear maps such that the following equations hold for all $x,y\in A$:
\begin{eqnarray}
    \mathcal{L}([x,y]_{A})&=&\mathcal{L}(x)\mathcal{L}(y)-\mathcal{L}(y)\mathcal{L}(x),\label{eq:Leibniz rep1}\\
    \mathcal{R}([x,y]_{A})&=&\mathcal{L}(x)\mathcal{R}(y)-\mathcal{R}(y)\mathcal{L}(x),\label{eq:Leibniz rep2}\\
    \mathcal{R}(y)\mathcal{L}(x)&=&-\mathcal{R}(y)\mathcal{R}(x).\label{eq:Leibniz rep3}
\end{eqnarray}
Two representations $( \mathcal{L}_{1},\mathcal{R}_{1},V_{1} )$
and $( \mathcal{L}_{2},\mathcal{R}_{2},V_{2} )$ are called {\bf equivalent} if there exists a linear isomorphism $f:V_{1}\rightarrow V_{2}$ such that
\begin{equation*}
    f(\mathcal{L}_{1}(x)v)=\mathcal{L}_{2}(x)f(v),\;
    f(\mathcal{R}_{1}(x)v)=\mathcal{R}_{2}(x)f(v),\;
    \forall x\in A, v\in V_{1}.
\end{equation*}
\end{defi}

Let $(A,[-,-]_{A})$ be a Leibniz algebra, $V$  be a vector space and
    $\mathcal{L},\mathcal{R}:A\rightarrow\mathrm{End}(V)$ be linear maps. Then $(\mathcal{L},\mathcal{R},V)$ is a representation of the Leibniz algebra $(A,[-,-]_{A})$ if and only if the direct sum $A\oplus V$ of vector spaces is turned into a \textbf{(semi-direct product)} Leibniz algebra by defining the multiplication on $A\oplus V$ by
    \begin{equation}
        [x+u,y+v]=[x,y]_{A}+\mathcal{L}(x)v+\mathcal{R}(y)u,
        \;\;\forall x,y\in A, u,v\in V.
    \end{equation}
    %We denote this Leibniz algebra structure by $A\ltimes_{\mathcal{L},\mathcal{R}}V$.

\begin{ex}
    Let $(A,[-,-]_{A})$ be a Leibniz algebra. Then $(\mathcal{L}_{A},\mathcal{R}_{A},A)$ is a representation of $(A,[-,-]_{A})$.
\end{ex}

Let $A$ and $V$ be vector spaces. For a linear map $f:A\rightarrow\mathrm{End}(V)$, we set a linear map $f^{*}:A\rightarrow\mathrm{End}(V^{*})$ by
\begin{equation}
    \langle f^{*}(x)u^{*}, v\rangle=-\langle u^{*},f(x)v\rangle,\;\;\forall x\in A, u^{*}\in V^{*},v\in V,
\end{equation}
where $\langle-,-\rangle$ is the ordinary pair between $V$ and $V^{*}$.

If $(\mathcal{L},\mathcal{R},V)$ is a representation of a Leibniz algebra $(A,[-,-]_{A})$, then it is known in \cite{ST} that $(\mathcal{L}^{*},-\mathcal{L}^{*}-\mathcal{R}^{*},V^{*})$ is also a representation of $(A,[-,-]_{A})$. In particular, $(\mathcal{L}_{A}^{*},-\mathcal{L}_{A}^{*}-\mathcal{R}_{A}^{*},A^{*})$ is a  representation of $(A,[-,-]_{A})$.

\begin{pro}
    Let $(A,[-,-]_{A},\Delta_{r})$ be a factorizable Leibniz bialgebra.
    Then $T_{(r-\sigma(r))}$ is a linear isomorphism giving the equivalence of $(\mathcal{L}_{A},\mathcal{R}_{A},A)$ and $(\mathcal{L}^{*}_{A},-\mathcal{L}^{*}_{A}-\mathcal{R}^{*}_{A},A^{*})$ as representations of $(A,[-,-]_{A})$.
\end{pro}
\begin{proof}
    Let $x\in A, a^{*},b^{*}\in A^{*}$. Then we have
    \begin{eqnarray*}
        \langle( (\mathcal{L}_{A}+\mathcal{R}_{A})(x)\otimes\mathrm{id})  (r-\sigma(r)), a^{*}\otimes b^{*}\rangle&=&-\langle (r-\sigma(r)), (\mathcal{L}^{*}_{A}+\mathcal{R}^{*}_{A})(x)a^{*}\otimes b^{*}\rangle\\
        &=&-\langle T_{(r-\sigma(r))}( (\mathcal{L}^{*}_{A}+\mathcal{R}^{*}_{A})(x)a^{*}  ), b^{*}\rangle,\\
        \langle(\mathrm{id}\otimes\mathcal{R}(x))(r-\sigma(r)),a^{*}\otimes b^{*}\rangle&=&-\langle(r-\sigma(r)), a^{*}\otimes \mathcal{R}^{*}_{A}(x) b^{*}\rangle\\
        &=&-\langle T_{(r-\sigma(r))}(a^{*}),\mathcal{R}^{*}_{A}(x) b^{*}\rangle\\
        &=&\langle \mathcal{R}_{A}(x) T_{(r-\sigma(r))}(a^{*}), b^{*}\rangle.
    \end{eqnarray*}
Then by Eq.~\eqref{s.i.}, we have
\begin{equation*}
T_{(r-\sigma(r))}((-\mathcal{L}^{*}_{A}-\mathcal{R}^{*}_{A})(x)a^{*})=\mathcal{R}_{A}(x)T_{(r-\sigma(r))}(a^{*}).
\end{equation*}
On the other hand, by Eq.~\eqref{s.i.}, we also have
\begin{equation*}
(\mathcal{L}_{A}(x)\otimes\mathrm{id}+\mathrm{id}\otimes\mathcal{L}_{A}(x))(r-\sigma(r))=0,
\end{equation*}
which gives
\begin{equation*}
    T_{(r-\sigma(r))}(\mathcal{L}^{*}_{A}(x)a^{*})=\mathcal{L}_{A}(x)T_{(r-\sigma(r))}(a^{*}).
\end{equation*}
Hence the conclusion follows.
\end{proof}}

Let $((A,[-,-]_{A}),(A^{*},[-,-]_{A^{*}}))$ be a Leibniz bialgebra. By \cite{ST}, there is a resulting Leibniz algebra structure on the double space $D=A\oplus A^{*}$ given by
    \begin{eqnarray*}
        [x+a^{*},y+b^{*}]_{D}&=&[x,y]_{A}+\mathcal{L}^{*}_{A^{*}}(a^{*})y-(\mathcal{L}^{*}_{A^{*}}+\mathcal{R}^{*}_{A^{*}})(b^{*})x\\
&&+ [a^{*},b^{*}]_{A^{*}}+\mathcal{L}^{*}_{A}(x)b^{*}-(\mathcal{L}^{*}_{A}+\mathcal{R}^{*}_{A})(y)a^{*},
    \end{eqnarray*}
for all $x,y\in A, a^{*}, b^{*}\in A^{*}$, which is called the \textbf{double Leibniz algebra}. Then we have the following result, which serves as a complement of Lemma \ref{lem:fact}.

\begin{pro}
    Let $(A,[-,-]_{A},\Delta_{r})$ be a factorizable Leibniz algebra.
    Then the double Leibniz algebra $(D,[-,-]_{D})$ is isomorphic to the direct sum $A\oplus A$ of Leibniz algebras.
\end{pro}
\begin{proof}
    Define a linear map $\theta:D=A\oplus A^{*}\rightarrow A\oplus A$  by
    \begin{equation*}
        \theta(x)=(x,x), \;\theta(a^{*})=(T_{r}(a^{*}),T_{\sigma(r)}(a^{*})),\;\forall x\in A, a^{*}\in A^{*}.
    \end{equation*}
Since $T_{r}-T_{\sigma(r)}$ is a linear isomorphism, we deduce that $\theta$ is a bijective linear map.
It is clear that $\theta|_{A}$ is an embedding of Leibniz algebras. By Corollary \ref{cor:homo}, $\theta|_{A^{*}}$ is a homomorphism of Leibniz algebras.
Moreover, for all $x\in A, a^{*}\in A^{*}$, we have
\begin{eqnarray*}
    \theta([x,a^{*}]_{D})&=&\theta( -(\mathcal{L}^{*}_{A^{*}}+\mathcal{R}^{*}_{A^{*}})(a^{*})x+\mathcal{L}^{*}_{A}(x)a^{*} )\\
    &=&(T_{r}(\mathcal{L}^{*}_{A}(x)a^{*})-(\mathcal{L}^{*}_{A^{*}}+\mathcal{R}^{*}_{A^{*}})(a^{*})x,T_{\sigma(r)}(\mathcal{L}^{*}_{A}(x)a^{*})-(\mathcal{L}^{*}_{A^{*}}+\mathcal{R}^{*}_{A^{*}})(a^{*})x).
\end{eqnarray*}
Thus for all $b^{*}\in A^{*}$, we have
\begin{eqnarray*}
\langle T_{r}(\mathcal{L}^{*}_{A}(x)a^{*}),b^{*}\rangle&=&\langle r,\mathcal{L}^{*}_{A}(x)a^{*}\otimes b^{*}\rangle=-\langle (\mathcal{L}_{A}(x)\otimes\mathrm{id})r, a^{*}\otimes b^{*}\rangle,\\
-\langle\mathcal{L}^{*}_{A^{*}}(a^{*})x,b^{*}\rangle&=&\langle x,[a^{*},b^{*}]_{r}\rangle=\langle\Delta_{r}(x),a^{*}\otimes b^{*}\rangle\\
&=&\langle ((\mathcal{L}_{A}+\mathcal{R}_{A})(x)\otimes\mathrm{id}-\mathrm{id}\otimes\mathcal{R}_{A}(x))r, a^{*}\otimes b^{*}\rangle,\\
-\langle\mathcal{R}^{*}_{A^{*}}(a^{*})x,b^{*}\rangle&=&\langle x,[b^{*},a^{*}]_{r}\rangle=\langle ((\mathcal{L}_{A}+\mathcal{R}_{A})(x)\otimes\mathrm{id}-\mathrm{id}\otimes\mathcal{R}_{A}(x))r, b^{*}\otimes a^{*}\rangle\\
&\overset{\eqref{s.i.}}{=}&\langle ((\mathcal{L}_{A}+\mathcal{R}_{A})(x)\otimes\mathrm{id}-\mathrm{id}\otimes\mathcal{R}_{A}(x))\sigma(r), b^{*}\otimes a^{*}\rangle\\
&=&\langle (\mathrm{id}\otimes(\mathcal{L}_{A}+\mathcal{R}_{A})(x)-\mathcal{R}_{A}(x)\otimes \mathrm{id})r, a^{*}\otimes b^{*}\rangle.
\end{eqnarray*}
Hence we have
\begin{eqnarray*}
    \langle T_{r}(\mathcal{L}^{*}_{A}(x)a^{*})-(\mathcal{L}^{*}_{A^{*}}+\mathcal{R}^{*}_{A^{*}})(a^{*})x, b^{*}\rangle&=&\langle(\mathrm{id}\otimes\mathcal{L}_{A}(x))r,a^{*}\otimes b^{*}\rangle=-\langle r,a^{*}\otimes\mathcal{L}^{*}_{A}(x)b^{*}\rangle\\
    &=&-\langle T_{r}(a^{*}),\mathcal{L}^{*}_{A}(x)b^{*}\rangle=\langle [x,T_{r}(a^{*})]_{A},b^{*}\rangle,
\end{eqnarray*}
that is,
\begin{equation*}
T_{r}(\mathcal{L}^{*}_{A}(x)a^{*})-(\mathcal{L}^{*}_{A^{*}}+\mathcal{R}^{*}_{A^{*}})(a^{*})x=[x,T_{r}(a^{*})]_{A}.
\end{equation*}
In the same way, we have
\begin{equation*}
    T_{\sigma(r)}(\mathcal{L}^{*}_{A}(x)a^{*})-(\mathcal{L}^{*}_{A^{*}}+\mathcal{R}^{*}_{A^{*}})(a^{*})x=[x,T_{\sigma(r)}(a^{*})]_{A}.
\end{equation*}
Thus we have
\begin{equation*}
    \theta([x,a^{*}]_{D})=([x,T_{r}(a^{*})]_{A},[x,T_{\sigma(r)}(a^{*})]_{A})=[\theta(x),\theta(a^{*})],
\end{equation*}
and similarly
\begin{equation*}
\theta([a^{*},x]_{D})=[\theta(a^{*}),\theta(x)].
\end{equation*}
In conclusion, $\theta:D\rightarrow A\oplus A$ is an isomorphism of Leibniz algebras. The proof is finished.
\end{proof}

%In the following,
%we show that there is moreover a factorizable Leibniz bialgebra structure on the Leibniz algebra $(D,[-,-]_{D})$.

In the following, we show that there is a factorizable Leibniz bialgebra structure on the double space of an arbitrary Leibniz bialgebra.

\begin{thm}
    Let $((A,[-,-]_{A}),(A^{*},[-,-]_{A^{*}}))$ be a Leibniz bialgebra. Suppose that
 $\{e_{1}$, $\cdots$, $e_{n}\}$ is a basis of $A$ and  $\{e^{*}_{1},\cdots,e^{*}_{n}\}$ is the dual basis. Set $$r=\sum\limits_{i=1}^{n}e_{i}\otimes e^{*}_{i}\in A\otimes A^{*}\subset D\otimes D.$$ Then $(D,[-,-]_{D},\Delta_{r})$ with $\Delta_r$ defined by  Eq.~\eqref{eq:cob} is a factorizable Leibniz bialgebra.
\end{thm}

%$(D,[-,-]_{D},\Delta_{r})$ is called the \textbf{Drinfeld classical double} of $((A,[-,-]_{A}),(A^{*},[-,-]_{A^{*}}))$.

\begin{proof}
    We first prove $(D,[-,-]_{D},\Delta_{r})$ is a quasi-triangular Leibniz bialgebra.
    For $x\in A$, we have
    \begin{eqnarray*}
        F(x)(r-\sigma(r))&=&\sum_{i}(-e_{i}\otimes[e^{*}_{i},x]_{D}+e^{*}_{i}\otimes[e_{i},x]_{A}+[x,e_{i}]_{A}\otimes e^{*}_{i}\\
        &&+[e_{i},x]_{A}\otimes e^{*}_{i}-[x,e^{*}_{i}]_{D}\otimes e_{i}-[e^{*}_{i},x]_{D}\otimes e_{i})\\
        &=&\sum_{i}(-e_{i}\otimes\mathcal{L}^{*}_{A^{*}}(e^{*}_{i})x+e_{i}\otimes(\mathcal{L}^{*}_{A}+\mathcal{R}^{*}_{A})(x)e^{*}_{i}+e^{*}_{i}\otimes[e_{i},x]_{A}\\
        &&+[x,e_{i}]_{A}\otimes e^{*}_{i}+[e_{i},x]_{A}\otimes e^{*}_{i}+\mathcal{R}^{*}_{A}(x)e^{*}_{i}\otimes e_{i}+\mathcal{R}^{*}_{A^{*}}(e^{*}_{i})x\otimes e_{i}).
    \end{eqnarray*}
    Observing that
    \begin{eqnarray*}
        &&\sum_{i}(-e_{i}\otimes\mathcal{L}^{*}_{A^{*}}(e^{*}_{i})x+\mathcal{R}^{*}_{A^{*}}(e^{*}_{i})x\otimes e_{i})=0,\\
        &&\sum_{i}(e_{i}\otimes(\mathcal{L}^{*}_{A}+\mathcal{R}^{*}_{A})(x)e^{*}_{i}+[e_{i},x]_{A}\otimes e^{*}_{i}+[x,e_{i}]_{A}\otimes e^{*}_{i})=0,\\
        &&\sum_{i}(e^{*}_{i}\otimes[e_{i},x]_{A}+\mathcal{R}^{*}_{A}(x)e^{*}_{i}\otimes e_{i})=0,
    \end{eqnarray*}
    we finally get $F(x)(r-\sigma(r))=0$. By duality, we also have $F(a^{*})(r-\sigma(r))=0,$ for all $a^{*}\in A^{*}$. Thus $r$ satisfies Eq.~\eqref{s.i.}.
    Furthermore,
    \begin{eqnarray*}
        [[r,r]]&=&\sum_{i,j}([e_{i},e_{j}]_{A}\otimes e^{*}_{i}\otimes e^{*}_{j}-e_{i}\otimes ([e_{j},e^{*}_{i}]_{D}+[e^{*}_{i},e_{j}]_{D})\otimes e^{*}_{j}+e_{i}\otimes e_{j}\otimes [e^{*}_{j},e^{*}_{i}]_{A^{*}})\\
        &=& \sum_{i,j}([e_{i},e_{j}]_{A}\otimes e^{*}_{i}\otimes e^{*}_{j}-e_{i}\otimes(\mathcal{L}^{*}_{A}(e_{j})e^{*}_{i}-(\mathcal{L}^{*}_{A^{*}}+\mathcal{R}^{*}_{A^{*}})(e^{*}_{i})e_{j}  )\otimes e^{*}_{j}\\
        &&-e_{i}\otimes(\mathcal{L}^{*}_{A^{*}}(e^{*}_{i})e_{j}-(\mathcal{L}^{*}_{A}+\mathcal{R}^{*}_{A})(e_{j})e^{*}_{i} )\otimes e^{*}_{j}+e_{i}\otimes e_{j}\otimes[e^{*}_{j},e^{*}_{i}]_{A^{*}} )\\
        &=&0,
    \end{eqnarray*}
    that is, $r$ is a solution of the CLYBE in $(D,[-,-]_{D})$.
    Hence $(D,[-,-]_{D},\Delta_{r})$ is a quasi-triangular Leibniz bialgebra.
    Moreover, the linear maps $T_{r},T_{\sigma(r)}:D^{*}\rightarrow D$ are respectively given by
    \begin{equation*}
        T_{r}(a^{*}+x)=a^{*},\;\; T_{\sigma(r)}(a^{*}+x)=x,\;\;\forall x\in A, a^{*}\in A^{*},
    \end{equation*}
which implies that $T_{r-\sigma(r)}(a^{*}+x)=a^{*}-x$, and   $T_{r-\sigma(r)}$ is a linear isomorphism. Hence $(D,[-,-]_{D},\Delta_{r})$ is a factorizable Leibniz bialgebra.
\end{proof}

\section{Relative Rota-Baxter operators and skew-symmetric quadratic Rota-Baxter Leibniz algebras}\label{sec4}

In this section, we further study  operator forms of solutions of
the CLYBE, which draw inspiration from \cite{BGN,KM}. We show that
relative Rota-Baxter operators with weights on Leibniz algebras
can be used to characterize solutions of the CLYBE whose
skew-symmetric parts are invariant. {{On skew-symmetric quadratic
Leibniz algebras, such operators correspond to linear
transformations satisfying the Rota-Baxter type condition (see
Eq.~\eqref{eq:pro3,1} in Proposition \ref{pro1-6}). Moreover, the
notion of skew-symmetric quadratic Rota-Baxter Leibniz algebras is
introduced, which  give rise to triangular Leibniz bialgebras in
the case of weight $0$, and are in one-to-one correspondence with
factorizable Leibniz bialgebras in the case of nonzero weights. }}
\delete{ {{
        A solution $r\in A\otimes A$ of the CLYBE in a Leibniz algebra $(A,[-,-]_{A})$ can be characterized as a relative Rota-Baxter operator $T_{r}\in\mathrm{Hom}(A^{*},A)$ with a weight on $(A,[-,-]_{A})$.
        Supported from the linear isomorphism $\omega^{\sharp}\in \mathrm{Hom}(A,A^{*})$ given by a skew-symmetric quadratic Leibniz algebra $(A,[-,-]_{A},\omega)$, such operator $T_{r}$ leads to a Rota-Baxter type operator $\beta=T_{r}\omega^{\sharp}\in\mathrm{End}(A)$.
    Then we introduce the notion of skew-symmetric quadratic Rota-Baxter Leibniz algebras, which equivalently characterize the quadruple $(A,[-,-]_{A},\omega,\beta)$ when $r$ and $\omega$ satisfy further relations.
        Explicitly, skew-symmetric quadratic Rota-Baxter Leibniz algebras of weight $0$
        give  rise to triangular Leibniz bialgebras,
        whereas there is a one-to-one correspondence between skew-symmetric quadratic Rota-Baxter Leibniz algebras of nonzero weights and factorizable Leibniz bialgebras.}}

    {\gl{I spent too much space to explain these relations, but the effect is still not so well.}}
\delete{
We show that relative Rota-Baxter operators with weights on Leibniz algebras can be used to characterize
solutions of the CLYBE whose skew-symmetric parts are invariant. Such operators lead to the Rota-Baxter
type operators under some conditions. In particular, we introduce the notion of skew-symmetric quadratic Rota-Baxter Leibniz algebras, and show that a
skew-symmetric quadratic Rota-Baxter Leibniz algebra of weight $0$
gives rise to a triangular Leibniz bialgebra,
whereas there is a one-to-one correspondence between skew-symmetric quadratic Rota-Baxter Leibniz algebras of nonzero weights and factorizable Leibniz bialgebras.}
}

\subsection{Relative Rota-Baxter operators and the classical Leibniz Yang-Baxter equation}\label{sec4.1}\

We recall the notions of representations of Leibniz algebras and $A$-Leibniz algebras.

\begin{defi}
    A \textbf{representation} of a Leibniz algebra $(A,[-,-]_{A})$ is a triple $(\mathnormal{l},\mathnormal{r},V)$, where $V$ is a vector space and $\mathnormal{l},\mathnormal{r}:A\rightarrow\mathrm{End}(V)$ are linear maps such that the following equations hold:
    \begin{eqnarray}
        \mathnormal{l}([x,y]_{A})&=&\mathnormal{l}(x)\mathnormal{l}(y)-\mathnormal{l}(y)\mathnormal{l}(x),\label{eq:Leibniz rep1}\\
        \mathnormal{r}([x,y]_{A})&=&\mathnormal{l}(x)\mathnormal{r}(y)-\mathnormal{r}(y)\mathnormal{l}(x),\label{eq:Leibniz rep2}\\
        \mathnormal{r}(y)\mathnormal{l}(x)&=&-\mathnormal{r}(y)\mathnormal{r}(x),\label{eq:Leibniz rep3}
    \end{eqnarray}
for all $x,y\in A$.
\end{defi}

%\cm{I still feel that we might use $(l,r,V \mathtt{l}\mathtt{r})$ to denote a representation.}

Let $(A,[-,-]_{A})$ be a Leibniz algebra, $V$  be a vector space and
$\mathnormal{l},\mathnormal{r}:A\rightarrow\mathrm{End}(V)$ be linear maps. Then $(\mathnormal{l},\mathnormal{r},V)$ is a representation of the Leibniz algebra $(A,[-,-]_{A})$ if and only if the direct sum $A\oplus V$ of vector spaces is turned into a \textbf{(semi-direct product)} Leibniz algebra by defining the multiplication on $A\oplus V$ by
\begin{equation}
    [x+u,y+v]=[x,y]_{A}+\mathnormal{l}(x)v+\mathnormal{r}(y)u,
    \;\;\forall x,y\in A, u,v\in V.
\end{equation}
We denote this Leibniz algebra structure by $A\ltimes_{\mathnormal{l},\mathnormal{r}}V$.

\begin{ex}
    Let $(A,[-,-]_{A})$ be a Leibniz algebra. Then $(\mathcal{L}_{A},\mathcal{R}_{A},A)$ is a representation of $(A,[-,-]_{A})$.
\end{ex}

If $(\mathnormal{l},\mathnormal{r},V)$ is a representation of a Leibniz algebra $(A,[-,-]_{A})$, then it is known in \cite{ST} that $(\mathnormal{l}^{*},-\mathnormal{l}^{*}-\mathnormal{r}^{*},V^{*})$ is also a representation of $(A,[-,-]_{A})$. In particular, $(\mathcal{L}_{A}^{*},-\mathcal{L}_{A}^{*}-\mathcal{R}_{A}^{*},A^{*})$ is a  representation of $(A,[-,-]_{A})$.

\delete{
\begin{lem}\label{lem:4.1}
    Let $(A,[-,-]_{A})$ be a Leibniz algebra.
    Let $r\in A\otimes A$ and $x\in A, a^{*},b^{*}\in A^{*}$.
    Then the following conclusions hold:
    \begin{enumerate}
        \item $r$ is invariant if and only if the following equation holds:
        \begin{equation}\label{eq:lem1}
        [T_{r}(a^{*}),x]_{A}+T_{r}((\mathcal{L}^{*}_{A}+\mathcal{R}^{*}_{A})(x)a^{*})=0.
        \end{equation}
       \item If $r$ is skew-symmetric, then $r$ is invariant if and only if the following equation holds:
       \begin{equation}\label{eq:lem2}
        \mathcal{L}^{*}_{A}(T_{r}(a^{*}))b^{*}+
        (\mathcal{L}^{*}_{A}+\mathcal{R}^{*}_{A})(T_{r}(b^{*}))a^{*}=0.
       \end{equation}
    \end{enumerate}
    In particular, if $r$ is skew-symmetric and  invariant, then the following equations hold:
    \begin{eqnarray}
        \mathcal{R}^{*}_{A}(T_{r}(a^{*}))b^{*}-
        \mathcal{R}^{*}_{A}(T_{r}(b^{*}))a^{*}&=&0,\label{eq:lem3}\\
    (\mathcal{L}_{A}(x)\otimes\mathrm{id}+\mathrm{id}\otimes\mathcal{L}_{A}(x))r&=&0,\label{eq:lem4}\\
\left[x,T_{r}(a^{*})\right]_{A}-T_{r}(\mathcal{L}^{*}_{A}(x)a^{*})&=&0.\label{eq:lem5}
    \end{eqnarray}
\end{lem}
\begin{proof}
For all $x\in A, a^{*},b^{*}\in A^{*}$, we have
\begin{eqnarray*}
&&\langle [T_{r}(a^{*}),x]_{A}+T_{r}((\mathcal{L}^{*}_{A}+\mathcal{R}^{*}_{A})(x)a^{*}),b^{*}\rangle\\
&&=\langle r, -a^{*}\otimes\mathcal{R}^{*}_{A}(x)b^{*}+(\mathcal{L}^{*}_{A}+\mathcal{R}^{*}_{A})(x)a^{*}\otimes b^{*}\rangle\\
&&=\langle  (\mathrm{id}\otimes\mathcal{R}_{A}(x)-(\mathcal{L}_{A}+\mathcal{R}_{A})(x)\otimes\mathrm{id})r, a^{*}\otimes b^{*}\rangle\\
&&\overset{\eqref{eq:inv}}{=}0.
\end{eqnarray*}
Hence $r$ is invariant if and only if Eq.~\eqref{eq:lem1} holds.
Moreover, if $r$  is skew-symmetric, then we have
\begin{eqnarray*}
&&\langle \mathcal{L}^{*}_{A}(T_{r}(a^{*}))b^{*}+(\mathcal{L}^{*}_{A}+\mathcal{R}^{*}_{A})(T_{r}(b^{*}))a^{*},x\rangle\\
&&=-\langle b^{*},[T_{r}(a^{*}),x]_{A}\rangle-\langle a^{*},[T_{r}(b^{*}),x]_{A}+[x,T_{r}(b^{*})]_{A}\rangle\\
&&=\langle T_{r}(a^{*}),\mathcal{R}^{*}_{A}(x)b^{*}\rangle+\langle T_{r}(b^{*}),(\mathcal{L}^{*}_{A}+\mathcal{R}^{*}_{A})(x)a^{*}\rangle\\
&&=\langle r, a^{*}\otimes\mathcal{R}^{*}_{A}(x)b^{*}-(\mathcal{L}^{*}_{A}+\mathcal{R}^{*}_{A})(x)a^{*}\otimes b^{*}\rangle\\
&&=\langle  (-\mathrm{id}\otimes\mathcal{R}_{A}+(\mathcal{L}_{A}+\mathcal{R}_{A})(x)\otimes\mathrm{id})r, a^{*}\otimes b^{*}\rangle\\
&&\overset{\eqref{eq:inv}}{=}0.
\end{eqnarray*}
Hence $r$ is invariant if and only if Eq.~\eqref{eq:lem2} holds. In particular, if $r$ is skew-symmetric and  invariant, then by Eq.~\eqref{eq:lem2}, we get
Eq.~\eqref{eq:lem3}, which gives Eq.~\eqref{eq:lem4} and Eq.~\eqref{eq:lem5}.
\end{proof}}

Following the notion of $\mathfrak{g}$-Lie algebras \cite{BGN,KN}, we introduce the notion of $A$-Leibniz algebras.

 %\cm{This notion is not given the first time in \cite{BGN}, please check the more early references}
 %\gl{I did not find earlier references, but I found the paper "generalized NS algebras" has already cited \cite{BGN} as the original paper. So I guess it may be OK.}

\begin{defi}
    Let $(A,[-,-]_{A})$ and $(V,[-,-]_{V})$ be Leibniz algebras and $\mathnormal{l},\mathnormal{r}:A\rightarrow\mathrm{End}(V)$ be linear maps. If $(\mathnormal{l},\mathnormal{r},V)$ is a representation of $(A,[-,-]_{A})$ and the following equations hold:
    \begin{equation}\label{bim alg1}
        \mathnormal{l}(x)[u,v]_{V}=[\mathnormal{l}(x)u,v]_{V}+[u,\mathnormal{l}(x)v]_{V},
    \end{equation}
\begin{equation}\label{bim alg2}
[u,\mathnormal{l}(x)v]_{V}=[\mathnormal{r}(x)u,v]_{V}+\mathnormal{l}(x)[u,v]_{V},
\end{equation}
\begin{equation}\label{bim alg3}
[u,\mathnormal{r}(x)v]_{V}=\mathnormal{r}(x)[u,v]_{V}+[v,\mathnormal{r}(x)u]_{V},
\end{equation}
for all $x\in A, u,v\in V$, then we say the quadruple $(\mathnormal{l},\mathnormal{r},V,[-,-]_{V})$ is an \textbf{$A$-Leibniz algebra}.
\end{defi}

\begin{rmk}
    In \cite{Das}, the notion of an $A$-Leibniz algebra is also called a Leibniz $A$-representation.
    Similar to representations of Leibniz algebras, an $A$-Leibniz algebra also has an equivalent characterization in terms of the Leibniz algebra structure on the direct sum $A\oplus V$ of vector spaces. Here the multiplication is given by
    \begin{equation}
        [x+u,y+v]=[x,y]_{A}+\mathnormal{l}(x)v+\mathnormal{r}(y)u+[u,v]_{V},\;\;\forall x,y\in A, u,v\in V.
    \end{equation}
\end{rmk}

\begin{pro}\label{pro ss inv Leibniz}
    Let $(A,[-,-]_{A})$ be a Leibniz algebra and $r\in A\otimes A$ be skew-symmetric and invariant. Set a multiplication $\circ_{r}:A^{*}\otimes A^{*}\rightarrow A^{*}$ by
    \begin{equation}\label{eq:leib on A^{*}}
        a^{*}\circ_{r}b^{*}=\mathcal{L}^{*}_{A}(T_{r}(a^{*}))b^{*},\;\;\forall a^{*},b^{*}\in A^{*}.
    \end{equation}
Then $(\mathcal{L}^{*}_{A },-\mathcal{L}^{*}_{A }-\mathcal{R}^{*}_{A },A^{*},\circ_{r})$ is an $A$-Leibniz algebra.
\end{pro}
\begin{proof}
    Since $r$ is skew-symmetric and invariant,  for any $x,y\in A, a^{*},b^{*},c^{*}\in A^{*}$, we have
    \begin{eqnarray*}
    (\mathcal{L}_{A}(x)\otimes\mathrm{id})r&=&(\mathrm{id}\otimes\mathcal{R}_{A}(x)-\mathcal{R}_{A}(x)\otimes\mathrm{id})r=-(\mathrm{id}\otimes\mathcal{R}_{A}(x)-\mathcal{R}_{A}(x)\otimes\mathrm{id})\sigma(r)\\
    &=&-\sigma( (\mathcal{R}_{A}(x)\otimes\mathrm{id}-\mathrm{id}\otimes\mathcal{R}_{A}(x))r )=\sigma((\mathcal{L}(x)\otimes\mathrm{id})r)\\
    &=&(\mathrm{id}\otimes\mathcal{L}_{A}(x))\sigma(r)=-(\mathrm{id}\otimes\mathcal{L}_{A}(x))r,
    \end{eqnarray*}
that is,
    \begin{equation}\label{eq:lem4}
    (\mathcal{L}_{A}(x)\otimes\mathrm{id}+\mathrm{id}\otimes\mathcal{L}_{A}(x))r=0.
    \end{equation}
%\cm{why? It is not an obvious result.}
    Then
    \begin{eqnarray*}
    \langle [x,T_{r}(a^{*})]_{A},b^{*}\rangle&=&-\langle T_{r}(a^{*}),\mathcal{L}^{*}_{A}(x)b^{*}\rangle=-\langle r,a^{*}\otimes\mathcal{L}^{*}_{A}(x)b^{*}\rangle=\langle(\mathrm{id}\otimes\mathcal{L}_{A}(x))r, a^{*}\otimes b^{*}\rangle\\
    &\overset{\eqref{eq:lem4}}{=}&-\langle(\mathcal{L}_{A}(x)\otimes\mathrm{id})r, a^{*}\otimes b^{*}\rangle=\langle r, \mathcal{L}^{*}_{A}(x)a^{*}\otimes b^{*}\rangle=\langle T_{r}(\mathcal{L}^{*}_{A}(x)a^{*}),b^{*}\rangle.
    \end{eqnarray*}
That is,
\begin{equation}\label{eq:lem5}
\left[x,T_{r}(a^{*})\right]_{A}-T_{r}(\mathcal{L}^{*}_{A}(x)a^{*})=0.
\end{equation}
   Hence we have
    \begin{eqnarray*}
    &&
    a^{*}\circ_{r}( b^{*}\circ_{r}c^{*})-(a^{*}\circ_{r}b^{*})\circ_{r}c^{*}-b^{*}\circ_{r}( a^{*}\circ_{r}c^{*})\\
    &&=\mathcal{L}^{*}_{A}(T_{r}(a^{*}))\mathcal{L}^{*}_{A}(T_{r}(b^{*}))c^{*}-\mathcal{L}^{*}_{A}(T_{r}(\mathcal{L}^{*}_{A}(T_{r}(a^{*}))b^{*}))c^{*}-\mathcal{L}^{*}_{A}(T_{r}(b^{*}))\mathcal{L}^{*}_{A}(T_{r}(a^{*}))c^{*}\\
    &&\overset{\eqref{eq:lem5}}{=}\mathcal{L}^{*}_{A}(T_{r}(a^{*}))\mathcal{L}^{*}_{A}(T_{r}(b^{*}))c^{*}-\mathcal{L}^{*}_{A}([T_{r}(a^{*}),T_{r}(b^{*})]_{A})-\mathcal{L}^{*}_{A}(T_{r}(b^{*}))\mathcal{L}^{*}_{A}(T_{r}(a^{*}))c^{*}\\
    &&\overset{\eqref{eq:Leibniz rep1}}{=}0.
    \end{eqnarray*}
Thus $(A^{*},\circ_{r})$ is a Leibniz algebra. Moreover, we have
    \begin{eqnarray*}
        \langle\mathcal{L}^{*}_{A}(x)(a^{*}\circ_{r}b^{*}),y\rangle&=&-\langle \mathcal{L}^{*}_{A}(T_{r}(a^{*}))b^{*},[x,y]_{A}\rangle=\langle b^{*},[T_{r}(a^{*}),[x,y]_{A}]_{A}\rangle,\\
        \langle (\mathcal{L}^{*}_{A}(x)a^{*})\circ_{r}b^{*},y\rangle&=&\langle\mathcal{L}^{*}_{A}(T_{r}(\mathcal{L}^{*}_{A}(x)a^{*}))b^{*},y\rangle=-\langle b^{*},[T_{r}(\mathcal{L}^{*}_{A}(x)a^{*}),y]_{A}\rangle\\
        %=-\langle b^{*},[[x,T_{s}(a^{*})]_{A},y]_{A}\rangle,\\
        &\overset{\eqref{eq:lem5}}{=}&-\langle b^{*},[[x,T_{r }(a^{*})]_{A},y]_{A}\rangle=\langle b^{*},[[ T_{r}(a^{*}),x]_{A},y]_{A}\rangle,\\
        \langle a^{*}\circ_{r}(\mathcal{L}^{*}_{A}(x)b^{*}),y\rangle&=&\langle\mathcal{L}^{*}_{A}(T_{r}(a^{*}))\mathcal{L}^{*}_{A}(x)b^{*},y\rangle=\langle b^{*},[x,[T_{r}(a^{*}),y]_{A}]_{A}\rangle.
    \end{eqnarray*}
Hence Eq.~\eqref{bim alg1} holds for $(\mathcal{L}^{*}_{A },-\mathcal{L}^{*}_{A}-\mathcal{R}^{*}_{A },A^{*},\circ_{r})$. Similarly Eq.~\eqref{bim alg2} and Eq.~\eqref{bim alg3} hold for $(\mathcal{L}^{*}_{A },-\mathcal{L}^{*}_{A }-\mathcal{R}^{*}_{A },A^{*},\circ_{r})$,
and thus $(\mathcal{L}^{*}_{A },-\mathcal{L}^{*}_{A }-\mathcal{R}^{*}_{A },A^{*},\circ_{r})$ is an $A$-Leibniz algebra.
\end{proof}

We recall the notion of relative Rota-Baxter operators with weights on Leibniz algebras.

\begin{defi}\cite{Das}
    Let $(A,[-,-]_{A})$ be a Leibniz algebra and $(\mathnormal{l},\mathnormal{r},V,[-,-]_{V})$ be an $A$-Leibniz algebra. A linear map $T:V\rightarrow A$ is called a \textbf{relative Rota-Baxter operator of weight $\lambda\in\mathbb{K}$} on $(A,[-,-]_{A})$ with respect to $(\mathnormal{l},\mathnormal{r},V,[-,-]_{V})$ if $T$ satisfies
    \begin{equation}
        [Tu,Tv]_{A}=T(\mathnormal{l}(Tu)v+\mathnormal{r}(Tv)u+\lambda [u,v]_{V}),\;\;\forall u,v\in V.
    \end{equation}
\end{defi}

\begin{rmk}
If the multiplication $[-,-]_{V}$ on $V$ happens to be trivial, that is, $[u,v]_{V}=0$ for all $u,v\in V$, then
$T$ is simply called a \textbf{relative Rota-Baxter operator} on $(A,[-,-]_{A})$ with respect to $(\mathnormal{l},\mathnormal{r},V)$ as introduced in \cite{ST}.
On the other hand, recall that a {\bf Rota-Baxter operator of weight $\lambda\in\mathbb{K}$}  on a Leibniz algebra $(A,[-,-]_{A})$ is a linear map $\beta:A\rightarrow A$ such that
\begin{equation} \label{eq:rbo}
    [\beta(x),\beta(y)]_{A}=\beta([x,\beta(y)]_{A}+[\beta(x),y]_{A}+\lambda[x,y]_{A}), \quad \forall x, y\in A.
\end{equation}
Thus a Rota-Baxter operator on a Leibniz algebra $(A,[-,-]_{A})$ is a relative Rota-Baxter operator on $(A,[-,-]_{A})$ with respect to $(\mathcal{L}_{A},\mathcal{R}_{A},A,[-,-]_{A})$ of the same weight.
\end{rmk}

\begin{lem}\label{lem:4.1}
    Let $(A,[-,-]_{A})$ be a Leibniz algebra and $r\in A\otimes A$.
    Then  $r$ is invariant if and only if the following equation holds:
    \begin{equation}\label{eq:lem1}
        [T_{r}(a^{*}),x]_{A}+T_{r}((\mathcal{L}^{*}_{A}+\mathcal{R}^{*}_{A})(x)a^{*})=0,\;\forall x\in A, a^{*}\in A^{*}.
    \end{equation}
    If in addition $r$  is skew-symmetric, then $r$ is invariant if and only if the following equation holds:
    \begin{equation}\label{eq:lem2}
        \mathcal{L}^{*}_{A}(T_{r}(a^{*}))b^{*}+
        (\mathcal{L}^{*}_{A}+\mathcal{R}^{*}_{A})(T_{r}(b^{*}))a^{*}=0,\;\forall a^{*},b^{*}\in A^{*}.
    \end{equation}
\end{lem}
\begin{proof}
    For all $x\in A, a^{*},b^{*}\in A^{*}$, we have
    \begin{eqnarray*}
        &&\langle [T_{r}(a^{*}),x]_{A}+T_{r}((\mathcal{L}^{*}_{A}+\mathcal{R}^{*}_{A})(x)a^{*}),b^{*}\rangle\\
        &&=\langle r, -a^{*}\otimes\mathcal{R}^{*}_{A}(x)b^{*}+(\mathcal{L}^{*}_{A}+\mathcal{R}^{*}_{A})(x)a^{*}\otimes b^{*}\rangle\\
        &&=\langle  (\mathrm{id}\otimes\mathcal{R}_{A}(x)-(\mathcal{L}_{A}+\mathcal{R}_{A})(x)\otimes\mathrm{id})r, a^{*}\otimes b^{*}\rangle.
    \end{eqnarray*}
    Hence $r$ is invariant if and only if Eq.~\eqref{eq:lem1} holds.
    Moreover, if $r$  is skew-symmetric, then we have
    \begin{eqnarray*}
        &&\langle \mathcal{L}^{*}_{A}(T_{r}(a^{*}))b^{*}+(\mathcal{L}^{*}_{A}+\mathcal{R}^{*}_{A})(T_{r}(b^{*}))a^{*},x\rangle\\
        &&=-\langle b^{*},[T_{r}(a^{*}),x]_{A}\rangle-\langle a^{*},[T_{r}(b^{*}),x]_{A}+[x,T_{r}(b^{*})]_{A}\rangle\\
        &&=\langle T_{r}(a^{*}),\mathcal{R}^{*}_{A}(x)b^{*}\rangle+\langle T_{r}(b^{*}),(\mathcal{L}^{*}_{A}+\mathcal{R}^{*}_{A})(x)a^{*}\rangle\\
        &&=\langle r, a^{*}\otimes\mathcal{R}^{*}_{A}(x)b^{*}-(\mathcal{L}^{*}_{A}+\mathcal{R}^{*}_{A})(x)a^{*}\otimes b^{*}\rangle\\
        &&=\langle  (-\mathrm{id}\otimes\mathcal{R}_{A}(x)+(\mathcal{L}_{A}+\mathcal{R}_{A})(x)\otimes\mathrm{id})r, a^{*}\otimes b^{*}\rangle.
    \end{eqnarray*}
    Hence $r$ is invariant if and only if Eq.~\eqref{eq:lem2} holds.
\end{proof}

\begin{thm}\label{pro:rrb}
    Let $(A,[-,-]_{A})$ be a Leibniz algebra and $r\in A\otimes A$ satisfy Eq.~\eqref{s.i.}.
    Then the following conditions are equivalent:
    \begin{enumerate}
        \item\label{rrb0}
        $r$ is a solution of the CLYBE in $(A,[-,-]_{A})$ such that $(A,[-,-]_{A},\Delta_{r})$ with $\Delta_r$ defined by Eq.~\eqref{eq:cob} is a quasi-triangular Leibniz bialgebra.
        \item\label{rrb1} $T_{r}$ is a relative Rota-Baxter operator of weight $-1$ on $(A,[-,-]_{A})$ with respect to the  $A$-Leibniz algebra $(\mathcal{L}^{*}_{A},-\mathcal{L}^{*}_{A}-\mathcal{R}^{*}_{A}, A^{*},\circ_{r-\sigma(r)})$, where the multiplication $\circ_{r-\sigma(r)}$ is given by %Eq.~\eqref{eq:leib on A^{*}}, that is,
        \begin{equation}\label{eq:thm:3,1}
            a^{*}\circ_{r-\sigma(r)} b^{*}=\mathcal{L}^{*}_{A}( T_{r-\sigma(r)} (a^{*}))b^{*},\;\;\forall a^{*},b^{*}\in A^{*}.
        \end{equation}
    That is, the following equation holds:
    \begin{equation}\label{o-op-1}
        [T_{r}(a^{*}),T_{r}(b^{*})]_{A}=T_{r}(  \mathcal{L}^{*}_{A}(T_{r}(a^{*}))b^{*}- (\mathcal{L}^{*}_{A}+\mathcal{R}^{*}_{A})(T_{r}(b^{*}))a^{*}-a^{*}\circ_{r-\sigma(r)}b^{*}).
    \end{equation}
    \end{enumerate}
\end{thm}
\begin{proof}
 By Proposition \ref{pro ss inv Leibniz}, $(\mathcal{L}^{*}_{A},-\mathcal{L}^{*}_{A}-\mathcal{R}^{*}_{A}, A^{*},\circ_{r-\sigma(r)})$ is an $A$-Leibniz algebra. On the other hand, we have
 \begin{eqnarray*}
 &&T_{r}(\mathcal{L}^{*}_{A}(T_{r}(a^{*}))b^{*})-T_{r}((\mathcal{L}^{*}_{A}+\mathcal{R}^{*}_{A})(T_{r}(b^{*}))a^{*})-T_{r}(a^{*}\circ_{r-\sigma(r)} b^{*})\\
 &\overset{\eqref{eq:thm:3,1}}{=}&T_{r}(\mathcal{L}^{*}_{A}(T_{r}(a^{*}))b^{*})-T_{r}((\mathcal{L}^{*}_{A}+\mathcal{R}^{*}_{A})(T_{r}(b^{*}))a^{*})-T_{r}(\mathcal{L}^{*}_{A}(T_{r-\sigma(r)}(a^{*}))b^{*})\\
 &\overset{\eqref{eq:lem2}}{=}&T_{r}(\mathcal{L}^{*}_{A}(T_{r}(a^{*}))b^{*})-T_{r}((\mathcal{L}^{*}_{A}+\mathcal{R}^{*}_{A})(T_{r}(b^{*}))a^{*}) +T_{r}((\mathcal{L}^{*}_{A}+\mathcal{R}^{*}_{A})(T_{r-\sigma(r)}(b^{*}))a^{*})\\
 &=&T_{r}(\mathcal{L}^{*}_{A}(T_{r}(a^{*}))b^{*})-T_{r}((\mathcal{L}^{*}_{A}+\mathcal{R}^{*}_{A})(T_{\sigma(r)}(b^{*}))a^{*}).
 \end{eqnarray*}
 Thus Eq.~\eqref{o-op-1} holds if and only if Eq.~\eqref{eq:T homo1} holds.
 Hence the conclusion follows from Proposition \ref{pro:CLYBE1}.
\end{proof}

Taking a symmetric 2-tensor $r\in A\otimes A$ into Theorem~\ref{pro:rrb}, we recover the following result as implied in \cite{ST}.

\begin{cor}
    Let $(A,[-,-]_{A})$ be a Leibniz algebra and $r\in A\otimes A$ be symmetric. Then $r$ is a solution of the CLYBE in $(A,[-,-]_{A})$ if and only if $T_{r}:A^{*}\rightarrow A$ is a relative Rota-Baxter operator on  $(A,[-,-]_{A})$ with respect to the representation $(\mathcal{L}^{*}_{A},-\mathcal{L}^{*}_{A}-\mathcal{R}^{*}_{A},A^{*})$.
\end{cor}

%In the following, we study $\mathcal{O}$-operators with weights on quadratic Leibniz algebras, which lead to an analogue of Rota-Baxter operators.

Next we apply Theorem~\ref{pro:rrb} to the case of skew-symmetric quadratic Leibniz algebras.

Let $\omega$ be a nondegenerate bilinear form on a vector space $A$.
Then there is a linear isomorphism $\omega^{\sharp}:A\rightarrow A^{*}$ given by
\begin{equation*}
    \omega(x,y)=\langle\omega^{\sharp}(x),y\rangle,\;\;\forall x,y\in A.
\end{equation*}
Define a $2$-tensor $\phi_{\omega}\in A\otimes A$  to be the tensor form of $(\omega^{\sharp})^{-1}$, that is,
\begin{equation}\label{identify bf}
\langle\phi_{\omega},a^{*}\otimes b^{*}\rangle:=\langle (\omega^{\sharp})^{-1}a^{*},b^{*}\rangle,\;\;\forall a^{*},b^{*}\in A^{*}.
\end{equation}
%which equivalently gives
%$(\omega^{\sharp})^{-1}=(\omega^{\sharp})^{-1}$ via Eq.~\eqref{eq:id}.

\begin{lem}\label{lem3.10}
    Let $(A,[-,-]_{A})$ be a Leibniz algebra and $\omega$ be a nondegenerate bilinear form on $A$. Then  $(A,[-,-]_{A},\omega)$ is a skew-symmetric quadratic Leibniz algebra if and only if the corresponding $\phi_{\omega}\in A\otimes A$ via Eq.~\eqref{identify bf} is skew-symmetric and invariant.
\end{lem}
\begin{proof}
    It is obvious that $\omega$ is skew-symmetric if and only if $\phi_{\omega}$ is skew-symmetric.
     For any $x,y\in A$, we set $a^{*}=\omega^{\sharp}(x), b^{*}=\omega^{\sharp}(y)$.
    Under the skew-symmetric assumption, we have
\begin{eqnarray*}
    &&\omega(x,[y,z]_{A})-\omega([x,z]_{A}+[z,x]_{A},y)\\
    &&=\omega((\omega^{\sharp})^{-1}a^{*},[(\omega^{\sharp})^{-1}b^{*},z]_{A})-\omega([(\omega^{\sharp})^{-1}a^{*},z]_{A}+[z,(\omega^{\sharp})^{-1}a^{*}]_{A},(\omega^{\sharp})^{-1}b^{*})\\
    &&=\langle a^{*},[(\omega^{\sharp})^{-1}b^{*},z]_{A}\rangle+\langle b^{*},[(\omega^{\sharp})^{-1}a^{*},z]_{A}+[z,(\omega^{\sharp})^{-1}a^{*}]_{A}\rangle\\
    &&=-\langle\mathcal{L}^{*}_{A}((\omega^{\sharp})^{-1}b^{*})a^{*},z\rangle-\langle(\mathcal{L}^{*}_{A}+\mathcal{R}^{*}_{A})((\omega^{\sharp})^{-1}a^{*})b^{*},z\rangle.
\end{eqnarray*}
Thus $\omega$ is invariant if and only if
\begin{equation*}
\mathcal{L}^{*}_{A}((\omega^{\sharp})^{-1}b^{*})a^{*}+(\mathcal{L}^{*}_{A}+\mathcal{R}^{*}_{A})((\omega^{\sharp})^{-1}a^{*})b^{*}=0,
\end{equation*}
which equivalently gives the invariance of $\phi_{\omega}$ by Lemma \ref{lem:4.1}.
\end{proof}

%In the following proposition, we correspond solutions of the CLYBE  on skew-symmetric quadratic Leibniz algebras satisfying Eq.~\eqref{s.i.} with Rota-Baxter type %operators.

\begin{pro}\label{pro1-6}
    Let $(A,[-,-]_{A},\omega)$ be a skew-symmetric quadratic Leibniz algebra  and $r\in A\otimes A$ satisfy Eq.~\eqref{s.i.}.
    Define a linear map $\beta:A\rightarrow A$ by
    \begin{equation}\label{Br,Brt}
        \beta(x)=T_{r}\omega^{\sharp}(x),\;\forall x\in A.
    \end{equation}
    Then  $r$ is a solution of the CLYBE in $(A,[-,-]_{A})$ if and only if  $\beta$ satisfies
        \begin{equation}\label{eq:pro3,1}
            [\beta(x),\beta(y)]_{A}=\beta \big([\beta(x),y]_{A}+ [x,\beta (y)]_{A}- [x,T_{r-\sigma(r)} \omega^{\sharp}(y) ]_{A}\big),\;\;\forall x,y\in A.
        \end{equation}
\end{pro}
\begin{proof}
    For any $x,y\in A$, we set $a^{*}=\omega^{\sharp}(x), b^{*}=\omega^{\sharp}(y)$. By Lemma \ref{lem3.10}, $\phi_{\omega}$ is skew-symmetric and invariant. Then we have
    \begin{eqnarray*}
        [\beta(x),\beta(y)]_{A}&=&[T_{r}(a^{*}),T_{r}(b^{*})]_{A},\\
        \beta[\beta(x),y]_{A}&=&T_{r}\omega^{\sharp}[T_{r}(a^{*}),(\omega^{\sharp})^{-1}b^{*}]_{A}\overset{\eqref{eq:lem5}}{=}T_{r}(\mathcal{L}^{*}_{A}(T_{r}(a^{*}))b^{*}),\\
        \beta[x,\beta(y)]_{A}&=&T_{r}\omega^{\sharp}[(\omega^{\sharp})^{-1}a^{*},T_{r}(b^{*})]_{A}\overset{\eqref{eq:lem1}}{=}-T_{r}((\mathcal{L}^{*}_{A}+\mathcal{R}^{*}_{A})(T_{r}(b^{*}))a^{*}),\\
        -\beta[x,T_{r-\sigma(r)}\omega^{\sharp}(y)]_{A}&=&
        -T_{r}\omega^{\sharp}[(\omega^{\sharp})^{-1}a^{*},T_{r-\sigma(r)}(b^{*})]_{A}\\&\overset{\eqref{eq:lem1}}{=}&
        T_{r}((\mathcal{L}^{*}_{A}+\mathcal{R}^{*}_{A})(T_{r-\sigma(r)}(b^{*}))a^{*})\\
        &\overset{\eqref{eq:lem2}}{=}&-T_{r}(  \mathcal{L}^{*}_{A}( T_{r-\sigma(r)}(a^{*})   ) b^{*} )=-T_{r}(a^{*}\circ_{r-\sigma(r)}b^{*}).
    \end{eqnarray*}
    Hence Eq.~\eqref{eq:pro3,1} holds if and only if Eq.~\eqref{o-op-1} holds.
Thus the conclusion follows from Theorem~\ref{pro:rrb}.% (\ref{rrb0})$\Longleftrightarrow$(\ref{rrb1}).
\end{proof}

\subsection{Skew-symmetric quadratic Rota-Baxter Leibniz algebras, triangular Leibniz bialgebras and factorizable Leibniz bialgebras}\label{sec4.2}\

\begin{defi}
    A \textbf{skew-symmetric quadratic Rota-Baxter Leibniz algebra of weight $\lambda$} is a quadruple $(A,[-,-]_{A},\omega,\beta)$, such that $(A,[-,-]_{A},\omega)$ is a skew-symmetric quadratic Leibniz algebra, $\beta$ is a Rota-Baxter operator of weight $\lambda$ on $(A,[-,-]_{A})$ and the following condition holds:
        \begin{equation}\label{strong}
            \omega(\beta (x),y)+\omega(x,\beta (y))+\lambda\omega(x,y)=0,\;\;\forall x,y\in A.
        \end{equation}
\end{defi}

\begin{pro}\label{mir}
    Let
$(A,[-,-]_{A},\omega)$ be a skew-symmetric quadratic Leibniz algebra and $\beta:A\rightarrow A$ be a linear map.
Then $(A,[-,-]_{A},\omega,\beta)$ is a skew-symmetric quadratic Rota-Baxter Leibniz algebra of weight $\lambda$ if and only if $(A,[-,-]_{A},-\omega,-(\lambda\mathrm{id}+\beta))$ is a skew-symmetric quadratic Rota-Baxter Leibniz algebra of weight $\lambda$. %\cm{why? some explanation should be given. Such as if $R$ is a Rota-Baxter operator of weight $\lambda$, then $-\lambda{\rm id}-R$ is also a Rota-Baxter operator.}
\end{pro}
\begin{proof}
    It is well-known that $\beta$ is a Rota-Baxter operator of weight $\lambda$ if and only if $-(\lambda\mathrm{id}+\beta)$ is a Rota-Baxter operator of weight $\lambda$. For all $x,y\in A$, we have
    \begin{equation*}
    \omega(-(\lambda\mathrm{id}+\beta)x,y)+\omega(x,-(\lambda\mathrm{id}+\beta)y)+\lambda\omega(x,y)=-\omega(\beta(x),y)-\omega(x,\beta(y))-\lambda\omega(x,y).
    \end{equation*}
    Thus $\beta$ satisfies Eq.~\eqref{strong} if and only if $-(\lambda\mathrm{id}+\beta)$ satisfies Eq.~\eqref{strong}.
    Hence the conclusion follows.
\end{proof}

\begin{pro}\label{pro:4.2}
Let $\beta$ be a Rota-Baxter operator of   weight $\lambda$ on a Leibniz algebra $(A,[-,-]_{A})$.
Then $$(A\ltimes_{\mathcal{L}^{*}_{A},-\mathcal{L}^{*}_{A}-\mathcal{R}^{*}_{A}} A^{*},\omega_{p},\beta-( \beta+\lambda\mathrm{id}_{A} )^{*})$$ is a skew-symmetric quadratic Rota-Baxter Leibniz algebra of weight $\lambda$, where the bilinear form $\omega_{p}$ on $A\oplus A^{*}$ is given by
\begin{equation}\label{eq:quad2}
\omega_{p}(x+a^{*},y+b^{*})=\langle x,b^{*}\rangle-\langle a^{*},y\rangle,\;\forall x,y\in A, a^{*},b^{*}\in A^{*}.
\end{equation}
\end{pro}
\begin{proof}
    It follows from a straightforward computation.
\end{proof}

\begin{ex}\label{ex:4.3}
    Let $(\mathfrak{g},\{-,-\}_{\mathfrak{g}})$ be a Lie algebra. Let ${\rm ad}_{\frak g}: \frak g\rightarrow {\rm End}({\frak g})$ be
the linear map defined by ${\rm ad}_{\frak g}(x) (y)=\{x,y\}$ for all $x,y\in \frak g$. Then $(\mathrm{ad}^{*}_{\mathfrak{g}},\mathfrak{g}^{*})$ is a representation of the Lie algebra $(\mathfrak{g},\{-,-\}_{\mathfrak{g}})$. %, where
    %\begin{equation*}
    %   \langle \mathrm{ad}^{*}_{\mathfrak{g}}(x)a^{*},y\rangle
    %   =-\langle a^{*},\mathrm{ad}_{\mathfrak{g}}(x)y\rangle=-\langle a^{*},\{x,y\}_{\mathfrak{g}}\rangle, \;\;\forall x,y\in\mathfrak{g},a^{*}\in \mathfrak{g}^{*}.
    %\end{equation*}
Hence there is a Leibniz algebra structure on $\mathfrak{g}\oplus\mathfrak{g}^{*}$ given by
\begin{equation*}
    [x+a^{*},y+b^{*}]=\{x,y\}_{\mathfrak{g}}+\mathrm{ad}^{*}_{\mathfrak{g}}(x)b^{*},\;\;\forall x,y\in\mathfrak{g},a^{*},b^{*}\in \mathfrak{g}^{*},
\end{equation*}
which is called the {\bf hemisemidirect product} \cite{Kin} of $(\mathfrak{g},\{-,-\}_{\mathfrak{g}})$ with respect to  $(\mathrm{ad}^{*}_{\mathfrak{g}},\mathfrak{g}^{*})$. %If $(\mathfrak{g},\{-,-\}_{\mathfrak{g}})$ is regarded as a special Leibniz algebra, then it is exactly $\mathfrak{g}\ltimes_{\mathrm{ad}^{*}_{\mathfrak{g}},0}\mathfrak{g}^{*}$ as the semi-direct product Leibniz algebra.
On the other hand, $\beta=\mathrm{id}_{\mathfrak{g}}$ is a Rota-Baxter operator of weight $-1$ on the Leibniz algebra $(\mathfrak{g},\{-,-\}_{\mathfrak{g}})$.
By Proposition \ref{pro:4.2},
$(\mathfrak{g}\ltimes_{\mathrm{ad}^{*}_{\mathfrak{g}},0}\mathfrak{g}^{*},\omega_{p},\mathrm{id}_{\mathfrak{g}})$ is a skew-symmetric quadratic Rota-Baxter Leibniz algebra of weight $-1$.
\end{ex}

%\cm{The original expression above is quite confused since Lie and Leibniz algebras are mixed together.}

\begin{lem}\label{lem:bf}
    Let $A$ be a vector space and $\omega$ be a nondegenerate skew-symmetric bilinear form.
    Let $r\in A\otimes A,\;\lambda\in\mathbb{K}$, and $\beta$ be given by Eq.~\eqref{Br,Brt}.
    Then $r$satisfies
    \begin{equation}\label{rw}
        r-\sigma(r)=-\lambda\phi_{\omega}
    \end{equation}
    if and only if $\beta$  satisfies the following equation:
    \begin{equation}\label{eq:lem:bf2}
        \omega(\beta(x),y)+\omega(x,\beta(y))+\lambda\omega(x,y)=0,\;\;\forall x,y\in A.
    \end{equation}
\end{lem}
\begin{proof}
     For any $x,y\in A$, set $a^{*}=\omega^{\sharp}(x), b^{*}=\omega^{\sharp}(y)$, we  have
    \begin{eqnarray*}
        \omega(\beta (x),y)&=&-\omega(y,\beta (x))=-\langle \omega^{\sharp}(y),T_{r}\omega^{\sharp}(x)\rangle=-\langle r, a^{*}\otimes b^{*}\rangle,\\
        \omega(x,\beta (y))&=&\langle\omega^{\sharp}(x),T_{r}\omega^{\sharp}(y)\rangle=\langle r,b^{*}\otimes a^{*}\rangle=\langle\sigma(r),a^{*}\otimes b^{*}\rangle,\\
        \lambda\omega(x,y)&=&-\lambda\omega(y,x)=-\lambda\langle \omega^{\sharp}(y),(\omega^{\sharp})^{-1}(\omega^{\sharp}(x))\rangle=-\lambda\langle\phi_{\omega}, a^{*}\otimes b^{*}\rangle.
    \end{eqnarray*}
    Hence $r$ satisfies  Eq.~\eqref{rw} if and only if $\beta$ satisfies Eq.~\eqref{eq:lem:bf2}.
\end{proof}

%As a direct consequence, we show that a skew-symmetric quadratic Rota-Baxter Leibniz algebra  of weight 0 gives rise to a triangular Leibniz bialgebra.

%Taking a symmetric $r\in A\otimes A$   into Proposition \ref{pro1-6}, we get the following corollary, which shows that a quadratic Rota-Baxter Leibniz algebra  of weight 0 gives rise to a triangular Leibniz bialgebra.

\begin{cor}\label{pro:triangular Leib}
   Let $(A,[-,-]_{A},\omega,\beta)$ be a skew-symmetric quadratic Rota-Baxter Leibniz algebra of weight $0$.
   Then there is a triangular Leibniz bialgebra $(A,[-,-]_{A},\Delta_{r})$ with $\Delta_r$ defined by Eq.~\eqref{eq:cob},
   where $r\in A\otimes A$ is given through the operator form $T_{r}$ by
   Eq.~\eqref{Br,Brt}, that is,
   \begin{equation}\label{eq:thm:quadratic to fact}
    T_{r}(a^{*})=\beta((\omega^{\sharp})^{-1}a^{*}),\;\;\forall a^{*}\in A^{*}.
   \end{equation}
\end{cor}
\begin{proof}
    It follows from Proposition \ref{pro1-6} and Lemma \ref{lem:bf} by observing $r-\sigma(r)=0$.
\end{proof}

%Next we  show that there is a one-to-one correspondence between factorizable Leibniz bialgebras and quadratic Rota-Baxter Leibniz algebras of nonzero weights.

Next we establish a one-to-one correspondence between skew-symmetric quadratic Rota-Baxter Leibniz algebras of nonzero weights and factorizable Leibniz bialgebras.

\begin{thm}\label{thm:quadratic to fact}
    Let $(A,[-,-]_{A},\omega,\beta)$ be a skew-symmetric quadratic Rota-Baxter Leibniz algebra of weight $\lambda\neq 0$. Then there is a factorizable Leibniz bialgebra $(A,[-,-]_{A},\Delta_{r})$ with $\Delta_r$ defined by Eq.~\eqref{eq:cob}, where $r\in A\otimes A$ is given through the operator form $T_{r}$ by
    Eq.~\eqref{eq:thm:quadratic to fact}.
\end{thm}
\begin{proof}
By Lemma \ref{lem:bf}, Eq.~\eqref{rw} holds. Since $\omega$ is nondegenerate and invariant, $r$ satisfies Eq.~\eqref{s.i.} and $T_{r-\sigma(r)}=-\lambda (\omega^{\sharp})^{-1}$ is a linear isomorphism. Moreover, since $\beta$ is a Rota-Baxter operator of weight $\lambda$, Eq.~\eqref{eq:pro3,1} holds in Proposition \ref{pro1-6}, such that $[[r,r]]=0$. In conclusion, $r$ is a solution of the CLYBE satisfying Eq.~\eqref{s.i.} and $T_{r-\sigma(r)}$ is nondegenerate, and hence  $(A,[-,-]_{A},\Delta_{r})$ is a factorizable Leibniz bialgebra.
\end{proof}

\begin{cor}\label{cor:4.20}
Let $\beta$ be a Rota-Baxter operator of weight $\lambda\neq 0$ on a Leibniz algebra $(A,[-,-]_{A})$. Let $\{e_{1},\cdots,e_{n}\}$ be a basis of $A$ and $\{e^{*}_{1},\cdots,e^{*}_{n}\}$ be the dual basis.
Then $(A\ltimes_{\mathcal{L}^{*}_{A},-\mathcal{L}^{*}_{A}-\mathcal{R}^{*}_{A}} A^{*},\Delta_{r})$ is a factorizable Leibniz bialgebra with $\Delta_r$ defined by Eq.~\eqref{eq:cob}, where
\begin{equation}\label{eq:cor:4.20}
    r=\sum_{i}(\beta+\lambda)e_{i}\otimes e^{*}_{i}+e^{*}_{i}\otimes\beta(e_{i}).
\end{equation}
\end{cor}
\begin{proof}
By Proposition \ref{pro:4.2},   $(A\ltimes_{\mathcal{L}^{*}_{A},-\mathcal{L}^{*}_{A}-\mathcal{R}^{*}_{A}} A^{*},\omega_{p},\beta-( \beta+\lambda\mathrm{id}_{A} )^{*})$ is a skew-symmetric quadratic Rota-Baxter Leibniz algebra of weight $\lambda$.
For all $x\in A, a^{*}\in A^{*}$, we have
\begin{equation*}
    \omega_{p}^{\sharp}(x+a^{*})=x-a^{*},
\end{equation*}
and following Eq.~\eqref{eq:thm:quadratic to fact}, we obtain a linear transformation $T_{r}$ on $A\oplus A^{*}$ by
\begin{equation*}
    T_{r}(x+a^{*})=(\beta-( \beta+\lambda\mathrm{id}_{A} )^{*})(\omega_{p}^{\sharp})^{-1}(x+a^{*})=\beta(x)+(\beta+\lambda\mathrm{id}_{A})^{*}a^{*}.
\end{equation*}
Then we have
\begin{eqnarray*}
&&\sum_{i,j}\langle r,e_{i}\otimes e^{*}_{j}\rangle=\sum_{i,j}\langle T_{r}(e_{i}),e^{*}_{j} \rangle=\sum_{i,j}\langle \beta(e_{i}),e^{*}_{j} \rangle,\\
&&\sum_{i,j}\langle r,e^{*}_{i}\otimes e_{j} \rangle=\sum_{i,j}\langle T_{r}(e^{*}_{i}),e_{j} \rangle=\sum_{i,j}\langle (\beta+\lambda\mathrm{id}_{A})^{*}e^{*}_{i},e_{j} \rangle=\sum_{i,j}\langle e^{*}_{i},(\beta+\lambda\mathrm{id}_{A})e_{j} \rangle.
\end{eqnarray*}
Hence Eq.~\eqref{eq:cor:4.20} holds, and by Theorem \ref{thm:quadratic to fact},
$(A\ltimes_{\mathcal{L}^{*}_{A},-\mathcal{L}^{*}_{A}-\mathcal{R}^{*}_{A}} A^{*},\Delta_{r})$ is a factorizable Leibniz bialgebra.
\end{proof}

\begin{ex}
    Let $(\mathfrak{g},\{-,-\}_{\mathfrak{g}})$ be the 2-dimensional non-abelian
    Lie algebra defined with respect to a basis $ \{e_{1},e_{2} \} $ by the following nonzero product:
    \begin{equation*}
    \{ e_{1},e_{2} \}_{\mathfrak{g}}=e_{1}.
    \end{equation*}
Then the Leibniz algebra $\mathfrak{g}\ltimes_{\mathrm{ad}^{*}_{\mathfrak{g}},0}\mathfrak{g}^{*}$ given by the hemisemidirect product $(\mathfrak{g},\{-,-\}_{\mathfrak{g}})$ with respect to $(\mathrm{ad}^{*}_{\mathfrak{g}},\mathfrak{g}^{*})$ is defined by the following nonzero products:
\begin{equation*}
    [e_{1},e_{2}]=e_{1},\; [e_{2},e_{1}]=-e_{1},\;[e_{1},e^{*}_{1}]=-e^{*}_{2},\; [e_{2},e^{*}_{1}]=e^{*}_{1}.
\end{equation*}
Then by Example \ref{ex:4.3}, $(\mathfrak{g}\ltimes_{\mathrm{ad}^{*}_{\mathfrak{g}},0}\mathfrak{g}^{*},\omega_{p},\mathrm{id}_{\mathfrak{g}})$ is a skew-symmetric quadratic Rota-Baxter Leibniz algebra of weight $-1$.
By Corollary \ref{cor:4.20},
there is a factorizable Leibniz bialgebra $(\mathfrak{g}\ltimes_{\mathrm{ad}^{*}_{\mathfrak{g}},0}\mathfrak{g}^{*},\Delta_{r})$,
where
\begin{equation*}
    r=e^{*}_{1}\otimes e_{1}+e^{*}_{2}\otimes e_{2}.
\end{equation*}
Note this Leibniz bialgebra is isomorphic to the one in Example \ref{ex:2.13} by identifying $e^{*}_{1}$ with $e_{3}$ and $e^{*}_{2}$ with $e_{4}$.
\end{ex}

The following result gives the converse of Theorem \ref{thm:quadratic to fact}.

\begin{thm}\label{thm:fact to quadratic}
Let $(A,[-,-]_{A},\Delta_{r})$ be a factorizable Leibniz
bialgebra. Let $\omega$ be {{the}} bilinear form on $A$ given by
\begin{equation}\label{eq:fact to quadratic}
    \omega(x,y)=-\lambda\langle T_{r-\sigma(r)}^{-1}(x),y\rangle,\;\lambda\neq 0,\;\forall x,y\in A,
\end{equation}
and $\beta$ be {{the}} linear map given by Eq.~\eqref{Br,Brt}.
Then $(A,[-,-]_{A},\omega,\beta)$ is a skew-symmetric quadratic
Rota-Baxter Leibniz  algebra of weight $\lambda$.
\end{thm}
\begin{proof}
    By Eq.~\eqref{identify bf} and Eq.~\eqref{eq:fact to quadratic}, we have $\omega^{\sharp}=-\lambda T_{r-\sigma(r)}^{-1}$, and thus $T_{r-\sigma(r)}=-\lambda (\omega^{\sharp})^{-1}$.
    Since $r$ satisfies Eq.~\eqref{s.i.} and $T_{r-\sigma(r)}$ is nondegenerate, it follows that $\omega$ given by Eq.~\eqref{eq:fact to quadratic} is a nondegenerate skew-symmetric invariant bilinear form on $(A,[-,-]_{A})$. Moreover, taking $T_{r-\sigma(r)}=-\lambda (\omega^{\sharp})^{-1}$ into Eq.~\eqref{eq:pro3,1} in Proposition \ref{pro1-6}, we see that $\beta$ is a Rota-Baxter operator of weight $\lambda$. Again from $T_{r-\sigma(r)}=-\lambda (\omega^{\sharp})^{-1}$, Eq.~\eqref{rw} holds, and by Lemma \ref{lem:bf}, Eq.~\eqref{eq:lem:bf2} holds. Thus $(A,[-,-]_{A},\omega,\beta)$ is a  skew-symmetric quadratic Rota-Baxter Leibniz  algebra of weight $\lambda$.
\end{proof}

Combining Theorems \ref{thm:quadratic to fact} and \ref{thm:fact to quadratic}, for a Leibniz algebra $(A,[-,-]_{A})$, there is a factorizable Leibniz bialgebra $(A,[-,-]_{A},\Delta_{r})$ if and only if there is a skew-symmetric quadratic  Rota-Baxter Leibniz algebra $(A,[-,-]_{A},\omega,\beta)$ of weight $\lambda\neq 0$. Here the mutual relation is given by
\begin{equation}
r-\sigma(r)=-\lambda\phi_{\omega}, \;T_{r}=\beta (\omega^{\sharp})^{-1}.
\end{equation}
Thus there is a one-to-one correspondence between factorizable Leibniz bialgebras and skew-symmetric quadratic  Rota-Baxter Leibniz algebras of nonzero weights.

\begin{pro}
Let  $(A,[-,-]_{A},\Delta_{r})$ be a factorizable Leibniz bialgebra which corresponds to a  skew-symmetric quadratic Rota-Baxter Leibniz algebra  $(A,[-,-]_{A},\omega,\beta)$ of weight $\lambda\neq 0$ via Theorems \ref{thm:quadratic to fact} and \ref{thm:fact to quadratic}. Then the factorizable Leibniz bialgebra   $(A$,
$[-,-]_{A}$,
$\Delta_{\sigma(r)})$ corresponds to the skew-symmetric quadratic Rota-Baxter Leibniz algebra  $(A$,
$[-,-]_{A}$,
$-\omega$,
$-(\lambda\mathrm{id}+\beta))$ of weight $\lambda$. In conclusion, we have the following commutative diagram.

\begin{equation*}
    %\resizebox{0.98\textwidth}{!}{
        \xymatrix@C=3cm@R=2.5cm{
            (A,[-,-]_{A},\Delta_{r}) \ar@{<->}[r]^-{ %Proposition
                {\rm Cor.}~\ref{cor:fact pair}}%_-{ %Proposition
                %\lambda=0}
            \ar@<-1ex>@{->}[d]_-{ %Proposition
                {\rm Thm.}~\ref{thm:fact to quadratic}}
            & (A,[-,-]_{A},\Delta_{\sigma(r)})
            \ar@<-1ex>@{->}[d]_-{{\rm Thm.}~\ref{thm:fact to quadratic}}\\
            %^-{ %Proposition
                %{\rm Thm.}~\ref{RB to fact}} \\
            (A,[-,-]_{A},\omega,\beta)
            %\ar@{-}\ar@{-}
            \ar@{<->}[r]^-{{\rm Prop.}~\ref{mir}}
            \ar@<-1ex>[u]_-{{\rm Thm.}~\ref{thm:quadratic to fact}}
            &
            (A,[-,-]_{A},-\omega,-(\lambda\mathrm{id}+\beta))
            \ar@<-1ex>[u]_-{{\rm Thm.}~\ref{thm:quadratic to fact}}
        }
        %}
\end{equation*}
\end{pro}
\begin{proof}
    By Theorem \ref{thm:fact to quadratic},  $(A,[-,-]_{A},\Delta_{\sigma(r)})$ gives rise to a skew-symmetric quadratic Rota-Baxter Leibniz algebra $(A,[-,-]_{A},\omega',\beta')$ of weight $\lambda$, where
    \begin{equation}\label{omega'}
        \omega'(x,y)\overset{\eqref{eq:fact to quadratic}}{=}-\lambda\langle T_{\sigma(r)-r}^{-1}(x),y\rangle=-\omega(x,y),
    \end{equation}
    and
    \begin{eqnarray*}
        \beta'(x)&\overset{\eqref{Br,Brt}}{=}&
        T_{\sigma(r)}\omega'^{\sharp}(x)
        \overset{\eqref{omega'}}{=}-T_{\sigma(r)}\omega^{\sharp}(x)\\
        &=&\lambda T_{\sigma(r)}T_{r-\sigma(r)}^{-1}(x)=\lambda (T_{r}-T_{r-\sigma(r)})T_{r-\sigma(r)}^{-1}(x)=-\lambda x+\lambda T_{r}T_{r-\sigma(r)}^{-1}(x)\\
        &=&-\lambda x-  T_{r}\omega^{\sharp}(x)\overset{\eqref{Br,Brt}}{=}-(\lambda\mathrm{id}+\beta)x,
    \end{eqnarray*}
for all $x,y\in A$. Hence
\begin{equation*}
(A,[-,-]_{A},\omega',\beta')=(A,[-,-]_{A},-\omega,-(\lambda\mathrm{id}+\beta)).
\end{equation*} Similarly, we obtain the converse side that the skew-symmetric quadratic Rota-Baxter Leibniz algebra $(A,[-,-]_{A},-\omega,-(\lambda\mathrm{id}+\beta))$
of weight $\lambda$ gives rise to the factorizable Leibniz bialgebra $(A,[-,-]_{A},\Delta_{\sigma(r)})$ via Theorem \ref{thm:quadratic to fact}.
\end{proof}

\noindent{\bf Acknowledgements.} This work is supported by
NSFC (11922110, 11931009, 12371029, 12271265, 12261131498, 12326319), %Sino-Russian Mathematical Center
the Fundamental Research Funds for the Central
Universities and Nankai Zhide Foundation. This work was completed in part while R. Tang was visiting the Key Laboratory of Mathematics and Its Applications of Peking University and he wishes to thank this laboratory for excellent working conditions.  R. Tang also thanks Professor Xiaomeng Xu for the warm hospitality to him during his stay at Peking University and for valuable discussions.
The authors thank the referee for valuable suggestions to improve the paper.

\end{document}